\documentclass[smallextended]{svjour3}       
\smartqed  
\usepackage{mathptmx}      
\usepackage{latexsym}

\usepackage{amssymb}
\usepackage{amsmath}
\usepackage{color}
\usepackage[colorlinks=true, citecolor=black, filecolor=black, linkcolor=black, urlcolor=black]{hyperref}
\usepackage{amsfonts}
\usepackage{verbatim}
\usepackage{simplewick}

\newcommand{\oneleg}{\Psi}

\newcommand{\Phiplus}{\Phi^+}
\newcommand{\Phiminus}{\Phi^-}

\newcommand{\bp}{{\bar\partial}}
\newcommand{\bs}{\bigskip}
\newcommand{\bfs}{\mathbf} 
\newcommand{\ms}{\medskip}

\newcommand{\pa}{\partial}
\newcommand{\sm}{\setminus}

\newcommand{\wh}{\widehat}
\newcommand{\wt}{\widetilde}
\newcommand{\ve}{\varepsilon}

\newcommand{\eff}{ \mathrm{eff}}
\newcommand{\hol}{ \mathrm{hol}}
\newcommand{\Aut}{ \mathrm{Aut}}

\newcommand{\id}{ \mathrm{id}}
\newcommand{\SLE}{ \mathrm{SLE}}
\newcommand{\scap}{ \mathrm{scap}}

\DeclareMathOperator{\Sing}{Sing}

\newcommand{\dd}{\mathrm{d}}

\newcommand{\FF}{\mathcal{F}}
\newcommand{\LL}{\mathcal{L}}
\newcommand{\NN}{\mathcal{N}}
\newcommand{\OO}{\mathcal{O}}
\newcommand{\PP}{\mathcal{P}}
\newcommand{\VV}{\mathcal{V}}
\newcommand{\XX}{\mathcal{X}}
\newcommand{\YY}{\mathcal{Y}}
\newcommand{\ZZ}{\mathcal{Z}}

\newcommand{\C}{\mathbb{C}}

\newcommand{\E}{\mathbf{E}}
\newcommand{\R}{\mathbb{R}}

\newcommand\arXiv[1]{\href{http://arxiv.org/abs/#1}{arXiv:#1}}
\newcommand\arxiv[1]{\href{http://arxiv.org/abs/#1}{arXiv: #1}}

\numberwithin{equation}{section}
\numberwithin{theorem}{section}
\newtheorem{lem}[theorem]{Lemma}
\newtheorem{prop}[theorem]{Proposition}

\def\subclassname{{\bfseries Mathematics Subject Classification
(2010)}\enspace}
\def\subclass#1{\par\addvspace\medskipamount{\rightskip=0pt plus1cm
\def\and{\ifhmode\unskip\nobreak\fi\ $\cdot$
}\noindent\subclassname\ignorespaces#1\par}}

\journalname{}

\begin{document}
\title{Conformal field theory of dipolar SLE with the Dirichlet boundary condition
\thanks{The authors were partially supported by NRF grant 2010-0021628. The first author also holds joint appointment in the Research Institute of Mathematics, Seoul National University.}
}
\titlerunning{CFT of dipolar SLE with the Dirichlet boundary condition}

\author{Nam-Gyu Kang \and Hee-Joon Tak}
\institute{Nam-Gyu Kang \email{nkang@snu.ac.kr} \at Department of Mathematical Sciences, Seoul National University, Seoul, 151-747, Republic of Korea
\and
Hee-Joon Tak \email{tdd502@snu.ac.kr} \at Department of Mathematical Sciences, Seoul National University, Seoul, 151-747, Republic of Korea}


\date{}

\maketitle

\begin{abstract}
We develop a version of dipolar conformal field theory based on the central charge modification of the Gaussian free field with the Dirichlet boundary condition and prove that correlators of certain family of fields in this theory are martingale-observables for dipolar SLE.
We prove the restriction property of dipolar SLE(8/3) and Friedrich-Werner's formula in the dipolar case.
\subclass{Primary 60J67, 81T40; Secondary 30C35}
\end{abstract}

\section{Introduction} \label{sec: intro}

We implement a version of dipolar conformal field theory with the Dirichlet boundary condition in a simply connected domain $D$ with two marked boundary points $q_-,q_+.$
Using this theory we study basic properties of a certain collection of dipolar $\SLE_\kappa$ martingale-observables.
In physics literature (e.g., \cite{BBH05}), it is well known that under the insertions of the boundary operator $\psi_{1;2}$ at $p\in\pa D$ and of the boundary operators $\psi_{0;1/2}$ at $q_\pm,$ all correlations of the fields in a certain collection are martingale-observables for  dipolar $\SLE_\kappa$ from $p$ to the boundary arc $Q$ ($p\notin Q$) with endpoints $q_\pm.$
We present its proof after we give a precise definition for $\Psi(p)\equiv\Psi(p;q_-,q_+)$ ($=\psi_{1;2}(p)\psi_{0;1/2}(q_-)\psi_{0;1/2}(q_+)$ up to boundary puncture operators) as a boundary vertex field rooted at $q_\pm$ of a single variable $p.$

A version of dipolar conformal field theory with the Neumann boundary condition can be implemented as the dual of theory with the Dirichlet boundary condition (see Subsection~\ref{ss: formal} below).
For example, as a bi-variant field, a Gaussian free field $\Phi^N(z,z_0)$ with the Neumann boundary condition can be defined as the dual boson $\widetilde\Phi(z,z_0)$ of the Gaussian free field $\Phi$ with the Dirichlet boundary condition.
On the other hand, in \cite{Kang12} we study a dipolar conformal field theory of central charge one with mixed boundary condition (Dirichlet boundary condition on one boundary arc and Neumann boundary condition on the other arc) and its relation to dipolar $\SLE_4.$

 Definitions and theories developed in the study of the chordal case \cite{KM11} and the radial case \cite{KM12} are modified into the dipolar case.
We also explain the similarities and differences of the chordal, the radial, and the dipolar conformal field theory.
For example, similarly as in the radial case, the neutrality condition is required for the (rooted) multi-vertex fields to be well-defined Fock space fields.
Under the neutrality condition, the (rooted) multi-vertex fields are $\Aut(D,q_\pm)$-invariant primary fields and their correlators with $\Psi(p)/\E\,[\Psi(p)]$ are dipolar $\SLE_\kappa$ martingale-observables.

\bs\section{Main results}

 \subsection{Dipolar SLE martingale-observables} \label{ss: SLE MO}
Let us consider a simply connected domain $(D,p,Q)$ with a marked boundary point $p$ and a marked boundary arc $Q \subseteq \pa D$ such that $p\notin \bar Q.$
We denote by $q_-,q_+$ two endpoints of $Q$ such that $q_-,p,q_+$ are positively oriented.
A dipolar Schramm-Loewner evolution ($\SLE_\kappa$) in $(D,p,Q)$ with a parameter $\kappa$ $(\kappa>0)$ is the conformally invariant law on random curves from the point $p$ to the arc $Q$ described by the solution $\psi_t(z)$ of the dipolar Loewner equation
\begin{equation} \label{eq: SLE}
\partial_t \psi_t(z) = \coth_2(\psi_t(z)-\xi_t), \quad (\xi_t = \sqrt\kappa B_t),
\end{equation}
where $\coth_2(z):=\coth(\frac12z)$ and $B_t$ is a one-dimensional standard Brownian motion with $B_0=0.$
As an initial data, $\psi_0:(D, p,Q)\to(\mathbb{S},0,\R + \pi i)$ is the conformal map from $D$ onto the strip $\mathbb{S}:=\{z \in\C\,|\, 0 < \Im\,z < \pi\}.$
(The solution $\psi_t(z)$ of \eqref{eq: SLE} exists up to a stopping time $\tau_z\in(0,\infty].$)
Then for all $t,$
$$w_t:(D_t,\gamma_t,Q)\to(\mathbb{S},0,\R+\pi i), \qquad w_t(z):=\psi_t(z)-\xi_t$$
is a conformal map from
$$D_t := \{z \in D: \tau_z>t\}$$
onto the strip $\mathbb{S}.$
The \emph{dipolar SLE curve} $\gamma$ is defined by
$$\gamma_t \equiv \gamma(t) := \lim_{z\to0} w_t^{-1}(z).$$
It is well known (\cite{SW05,Zhan08}) that the SLE trace exists in the dipolar case and that dipolar $\SLE_\kappa$ is equivalent to chordal $\SLE_\kappa(\frac12(\kappa-6),\frac12(\kappa-6))$ with two force points $q_\pm.$
We plan to apply definitions and constructions developed in \cite{KM11,KM12} to conformal field theory of $\SLE_\kappa(\bfs{\rho}).$
As a preliminary work, we study a version of dipolar conformal field theory in this paper.

To define dipolar SLE martingale-observables, let us recall the definition of non-random conformal fields.
See \cite[Section~4.1]{KM11} for more details.
A non-random \emph{conformal} field $M$ is an assignment of a (smooth) function
$(M\,\|\,\phi): ~\phi U\to\C$
to each local chart $\phi:U\to\phi U.$
A non-random conformal Fock space field $M$ is a $[\lambda,\lambda_*]$-\emph{differential} if for any two overlapping charts $\phi,\wt\phi,$ we have
$$(M\,\|\,\phi) = (h')^\lambda(\overline{h'})^{\lambda_*}(M\,\|\,\wt\phi)\circ h,$$
where $h=\wt\phi\circ\phi^{-1}:~ \phi(U\cap\wt U)\to \wt\phi(U\cap\wt U)$ is the transition map.
A pair $[\lambda,\lambda_*]$ is called degrees or conformal dimensions of $M.$
Non-random conformal Fock space fields $M$ are called \emph{pre-pre-Schwarzian forms}, \emph{pre-Schwarzian forms}, and \emph{Schwarzian forms} of order $\mu(\in\C)$ if the following transformation laws hold:
$$(M\,\|\,\phi)=(M\,\|\,\wt\phi)\circ h +\mu \log h',\, (M\,\|\,\phi)=h'\,(M\,\|\,\wt\phi)\circ h +\mu \frac{h''}{h'},$$
and $$(M\,\|\,\phi)=(h')^2\,(M\,\|\,\wt\phi)\circ h + \mu S_h,$$
respectively.
Here, $S_h$ is the Schwarzian derivative of $h,$
$S_h = ({h''}/{h'})' -\frac12({h''}/{h'})^2.$

A non-random (conformal) field $M$ of $n$ variables in $\mathbb{S}$ is said to be a \emph{martingale-observable} for dipolar $\SLE_\kappa$ if for any $z_1,\cdots ,z_n\in D,$ the process
$$M_t(z_1,\cdots, z_n)=(M_{D_t,\gamma_t,Q}\,\|\,\id)(z_1,\cdots, z_n) = (M\,\|\,w_t^{-1})(z_1,\cdots, z_n)$$
is a local martingale on dipolar SLE probability space.
(The process $M_t(z_1,\cdots, z_n)$ is stopped when any $z_j$ exits $D_t.$)
For example, we can use the identity chart of $D.$
Then for $[h,h_*]$-differentials $M$ with boundary conformal dimensions $h_\pm$ at $q_\pm,$ we have
$$M_t(z) = (w_t'(z))^h (\overline{w_t'(z)})^{h_{*}} (w_t'(q_-))^{h_-} (w_t'(q_+))^{h_+}M(w_t(z)).$$
If $M$ is a pre-Schwarzian form of order $\mu,$ then
$$M_t(z) = w_t'(z)M(w_t(z)) + \mu \frac{w_t''(z)}{w_t'(z)}.$$
Similarly, for a Schwarzian form $M$ of order $\mu,$ we have
$$M_t(z) = (w_t'(z))^2M(w_t(z)) + \mu S_{w_t}(z).$$

\subsection{A dipolar CFT} \label{ss: dipolar CFT}
All our fields in this paper are Fock space (correlational) fields constructed from the Gaussian free field $\Phi_{(0)}$ with the Dirichlet boundary condition, its derivatives, and Wick's exponentials $e^{\odot\alpha\Phi_{(0)}}\,(\alpha\in\C)$ by means of Wick's calculus.
(An alternate notation for Wick's exponentials of $\Phi_{(0)}$ is $:\!e^{\alpha\Phi_{(0)}}\!:.$)
For Fock space fields $X_1,\cdots,X_n$ and distinct points (\emph{nodes}) $z_1,\cdots,z_n$ in $D,$ a \emph{correlation function} 
$$\E[X_1(z_1)\cdots X_n(z_n)]$$
is defined by Wick's formula.
We will review its definition and basic properties in Subsection~\ref{ss: F-fields}.
For example, we define
$$\E[\Phi_{(0)}(z)]=0, \qquad \E[\Phi_{(0)}(z_1) \Phi_{(0)}(z_2)]= 2G(z_1,z_2)$$
($G$ is the Green's function for $D$ with the Dirichlet boundary condition) and
$$\E[\Phi_{(0)}(z_1) \cdots\Phi_{(0)}(z_n)] = \sum \prod_k \E[\Phi_{(0)}(z_{i_k})\Phi_{(0)}(z_{j_k})],$$
where the sum is over all partitions of the set $\{1,\cdots,n\}$ into disjoint pairs $\{i_k,j_k\}.$

For a simply connected domain $D$ with a marked boundary arc $Q \subseteq \pa D,$ we consider a conformal map
$$w \equiv w_{D,Q}: (D,Q) \to (\mathbb{S},\R+\pi i),$$
from $D$ onto the strip $\mathbb{S}.$
For a fixed parameter $b\in\R,$ we define central charge modifications $\Phi\equiv\Phi_{(b)}$ of the Gaussian free field $\Phi_{(0)}$ by
$$\Phi_{(b)} = \Phi_{(0)}  -2b\,\arg w'$$
(cf. central charge modifications of the Gaussian free field in the chordal case, see e.g., \cite[Section~10.1]{KM11}).
We denote by $\FF_{(b)}$ the OPE family of $\Phi_{(b)}$, the algebra over $\C$ spanned by $1,\pa^j\bp^k \Phi_{(b)},$ and derivatives of the vertex fields $\pa^j\bp^k \VV^\alpha_{(b)}\, (\alpha\in\C)$ under OPE multiplications $*.$
We will recall the definition of operator product expansion and its basic properties in Subsection~\ref{ss: OPE}.
As in the chordal case (\cite[Sections~3.3, 10.2]{KM11}), vertex fields are defined as OPE-exponentials of the bosonic field $\Phi=\Phi_{(b)}:$
$$\VV^\alpha\equiv \VV^\alpha_{(b)}=e^{*\alpha\Phi}=\sum_{n=0}^\infty\frac{\alpha^n}{n!}\Phi^{*n}.$$

 As in the radial case (\cite{KM12}), we extend the OPE family $\FF_{(b)}$ to include the bi-variant chiral bosonic field $\Phiplus_{(b)}(z,z_0),$ its conjugate, and (derivatives of) chiral multi-vertex fields
$$\OO^{(\bfs\sigma,\bfs\sigma_*)}(\bfs{z}),\,\,(\bfs\sigma = (\sigma_1\cdots,\sigma_n),\,\bfs\sigma_* = (\sigma_{1*}\cdots,\sigma_{n*}),\,\bfs{z} = (z_1,\cdots,z_n),\,z_j\in D)$$
(its precise definition is given in Subsection~\ref{ss: O})
with the \emph{neutrality condition}
$$\sum_{j=1}^n (\sigma_j + \sigma_{j*}) = 0.$$
As a multivalued field, the bi-variant chiral bosonic field $\Phiplus_{(b)}(z,z_0)$ can be expressed in terms of 1-point formal field $\Phiplus_{(b)}$ as follows:
$$\Phiplus_{(b)}(z,z_0)=\Phiplus_{(b)}(z)-\Phiplus_{(b)}(z_0), \quad  \Phiplus_{(b)}(z)=\Phiplus_{(0)}(z)+ib\log\frac{w'(z)}{1-w(z)^2},$$
where $w$ is a conformal map from $(D,q_-,q_+)$ onto $(\mathbb{H},-1,1)$ and the formal 1-point field $\Phiplus_{(0)}$ (which can be interpreted  as a ``holomorphic" part of $\Phi_{(0)}$ in the sense that $\Phi_{(0)}(z)= 2\,\Re\,\Phiplus_{(0)}(z)$) has formal correlations
$$\E[\Phiplus_{(0)}(z)\Phiplus_{(0)}(z_0)] = \log\frac1{w(z)-w(z_0)},\quad \E[\Phiplus_{(0)}(z)\overline{\Phiplus_{(0)}(z_0)}] = \log(w(z)-\overline{w(z_0)}).$$
The chiral multi-vertex field $\OO^{(\bfs\sigma,\bfs\sigma_*)}(\bfs{z})$ can be interpreted as the OPE exponential of the formal field
$$i\sum \sigma_j \Phiplus_{(b)}(z_j)-\sigma_{j*}\Phiminus_{(b)}(z_j),$$
where $\Phiminus_{(b)} = \overline{\Phiplus_{(b)}}.$
It turns out that the formal fields $\OO^{(\bfs\sigma,\bfs\sigma_*)}$ are well-defined Fock space fields if and only if the neutrality condition holds.
Under the neutrality condition, multi-vertex fields $\OO^{(\bfs\sigma,\bfs\sigma_*)}$ are $\Aut(D,q_\pm)$-invariant primary fields.
See Subsection~\ref{ss: Lie} for the definition of conformal invariance.

In the chordal (radial) case, under the insertion of Wick's exponential
$$e^{\odot ia \Phiplus_{(0)}(p,q)} \,(p,q\in\pa D,p\ne q), \qquad \big(e^{\odot - a\, \Im\,\Phiplus_{(0)}(p,q)} \,(p\in\pa D, q\in D)\big),$$
with $a = \sqrt{2/\kappa}$ and $b=a(\kappa/4-1),$ all fields in the extended OPE family $\FF_{(b)}$ of $\Phi_{(b)}$ satisfy the ``field Markov property" with respect to chordal (radial) SLE filtration, respectively.
(See \cite{BB03,RBGW07} for the chordal case and \cite{BB04,Cardy04} for the radial case from the physics perspective, cf. \cite[Proposition~14.3]{KM11} and \cite[Theorem~1.1]{KM12}.)
The dipolar version of this theorem can be stated as follows.
(As we mentioned in Section~\ref{sec: intro}, special cases of the following theorem are well known in physics literature, e.g., \cite{BBH05}.)

\begin{theorem}\label{main X}
For the tensor product $X=X_1(z_1)\cdots X_n(z_n)$ of fields $X_j$ in the OPE family $\FF_{(b)}$ of $\Phi_{(b)},$ the non-random fields
\begin{equation} \label{eq: main X}
\E\,[e^{\odot i\frac 12a (\Phiplus_{(0)}(p,q_-)+ \Phiplus_{(0)}(p,q_+))} \, X] \qquad (a = \sqrt{2/\kappa},\, b=\sqrt{\kappa/8}-\sqrt{2/\kappa})
\end{equation}
are martingale-observables for dipolar $\SLE_\kappa.$
\end{theorem}

\subsection{Boundary condition changing operators and rooted vertex fields}
The boundary condition changing operator on Fock space functionals/fields is the insertion of
\begin{equation} \label{eq: BC}
e^{\odot \frac12 ia (\Phiplus_{(0)}(p,q_-)+ \Phiplus_{(0)}(p,q_+))}.
\end{equation}
While this field does not belong to the extended OPE family $\FF_{(b)},$ we further extend the OPE family to contain the multi-vertex fields $\OO^{(\bfs{\sigma},\bfs{\sigma_*};\sigma_-,\sigma_+)}$ rooted at two points $q_-,q_+$ with the neutrality condition $(\sigma_-\!+\sigma_+\!+ \sum_{j=1}^n (\sigma_j + \sigma_{j*}) = 0)$ so that the field~\eqref{eq: BC} is represented as
$$\frac{\Psi(p)}{\E\,[\Psi(p)]},\qquad \Psi\equiv\Psi(\cdot;q_-,q_+):=\OO^{(a,0;-\frac12a,-\frac12a)}\in\FF_{(b)}.$$
Similarly as in the radial case, the definition of rooted multi-vertex fields can be obtained by normalizing multi-vertex fields
$$\OO^{(\bfs{\sigma},\bfs{\sigma_*})}\OO^{(\sigma_-)}(\eta_-)\OO^{(\sigma_+)}(\eta_+)$$
and taking a limit as $(\eta_-,\eta_+)$ approaches $(q_-,q_+).$
This rooting procedure also gives rise to the definition of the \emph{normalized tensor product} of rooted vertex fields as
$$\OO^{(\bfs{\sigma},\bfs{\sigma_*};\sigma_-,\sigma_+)}\star \OO^{(\bfs{\tau,\tau_*};\tau_-,\tau_+)}=\OO^{(\bfs{\sigma}+\bfs{\tau},\bfs{\sigma_*}+\bfs{\tau_*};\sigma_-+\tau_-,\sigma_++\tau_+)},$$
where $\bfs{\sigma,\sigma_*\tau,\tau_*}$ are divisors, maps from $D$ to $\R$ which take the value $0$ at all but finitely many points.
Like the multi-vertex fields, the rooted multi-vertex field $\OO^{(\bfs\sigma,\bfs\sigma_*;\sigma_-,\sigma_+)}(\bfs{z})$ can be interpreted as the OPE exponential of the field
\begin{equation} \label{eq: rooted boson}
i\sigma_- \Phiplus_{(b)}(q_-) + i\sigma_+\Phiplus_{(b)}(q_+)+i\sum \sigma_j \Phiplus_{(b)}(z_j)-\sigma_{j*}\Phiminus_{(b)}(z_j),
\end{equation}
where $\Phiplus_{(b)}(q_\pm) = \Phiplus_{(0)}(q_\pm).$
In the above expression, we do not need the terms $\Phiminus_{(0)}(q_\pm)$ since $\Phiplus_{(0)}(q_\pm)$ and $\Phiminus_{(0)}(q_\pm)$ are not linearly independent.
Indeed, $\Phiminus_{(0)}(q_\pm)=-\Phiplus_{(0)}(q_\pm)$ because $\Phi_{(0)}(q_\pm) = 0.$
Under the neutrality condition, rooted vertex fields $\OO^{(\bfs\sigma,\bfs\sigma_*;\sigma_-,\sigma_+)}$ are $\Aut(D,q_\pm)$-invariant primary fields with conformal dimensions $[\bfs{h}, \bfs{h}_*; h_-,h_+]:$
$$h_j=\frac{\sigma_j^2}{2}-\sigma_j b, \quad h_{j*}=\frac{\sigma_{j*}^2}{2}-\sigma_{j*} b , \quad h_\pm=\frac{\sigma_\pm^2}{2}.$$
We extend $\FF_{(b)}$ by adding the generators, \eqref{eq: rooted boson} and the rooted vertex fields with the neutrality conditions.
We call this extended collection of fields the \emph{extended OPE family} of $\Phi_{(b)}.$
For a rooted vertex field $\OO\equiv\OO^{(\bfs{\sigma},\bfs{\sigma_*};\sigma_-,\sigma_+)}$ with the neutrality condition,  $\E\,[\Psi(p)\OO]$ is not well-defined because both $\Psi(p) \equiv \Psi(p; q_-,q_+)$ and $\OO$ have nodes at $q_\pm.$
However, Theorem~\ref{main X} can be modified for a rooted vertex field with the neutrality condition (cf. \cite[Theorem~1.2]{KM12} for its radial version).

\begin{theorem}\label{main O}
For a rooted vertex field $\OO\equiv\OO^{(\bfs{\sigma},\bfs{\sigma_*};\sigma_-,\sigma_+)}$ with the neutrality condition, the non-random field
$$
\frac{\E[\Psi(p)\star\OO]}{\E[\Psi(p)]}
$$ is a dipolar $\SLE_\kappa$ martingale-observable.
\end{theorem}

Theorems~\ref{main X} -- \ref{main O} can be extended to the fields in the extended OPE family of $\Phi_{(b)}.$

\subsection{Examples of dipolar SLE martingale-observables}
We use conformal field theory to present the proof for the restriction property of dipolar $\SLE_{8/3}:$
for a fixed compact hull $K$ (i.e., $K$ is a compact set such that $\mathbb{S}\sm K$ is a simply connected subdomain of $\mathbb{S}$) with $K \cap (\R + \pi i) = \emptyset,$
the dipolar $\SLE_{8/3}$ path in $(\mathbb{S},0,-\infty,\infty)$ conditioned to avoid $K$ has the same distribution as the dipolar  $\SLE_{8/3}$ path in $(\mathbb{S}\sm K,0,-\infty,\infty).$
It is equivalent to Theorem~\ref{restriction} below.
To state it, let us recall the definition of the strip capacity (e.g., see \cite{Zhan04}) of a compact hull $K.$
For a compact hull $K$ with $K \cap (\R + \pi i) = \emptyset,$
there is a unique conformal transformation from $(\mathbb{S} \sm K ,-\infty,\infty)$ onto $(\mathbb{S},-\infty,\infty)$ such that
$$\lim_{z\to\pm\infty} \psi_K(z) - z = \pm \, s$$
for some $s\ge 0.$
Here, $z\to\pm\infty$ means that $z\in \mathbb{S}$ and $\Re\, z \to\pm\infty.$
This $s$ is called the strip capacity of $K$ and denoted by $\scap(K).$

\begin{theorem}\label{restriction} For a given compact hull $K,$
\begin{equation} \label{eq: restriction}
\mathbb{P}(\SLE_{8/3} \textrm{ path  avoids }   K)
=\psi'_K(0)^\lambda e^{-2\mu\,\scap(K)}, \quad(\lambda = 5/8, \,\mu=5/96).
\end{equation}
\end{theorem}

In the half-plane uniformization, \eqref{eq: restriction} reads as
\begin{equation} \label{eq: restriction'}
\mathbb{P}(\SLE_{8/3} \textrm{ path  avoids }   K)
=\psi'_K(0)^\lambda (\psi'_K(-1)\psi'_K(1))^\mu,
\end{equation}
where $K$ is a fixed compact hull with $ \pa K \cap \R \subseteq (-1,1)\sm\{0\}$ and $\psi_K$ is the conformal transformation from $(\mathbb{H} \sm K ,-1,1)$ onto $(\mathbb{H},-1,1)$ such that $\psi'_K(-1)=\psi'_K(1).$
Compare \eqref{eq: restriction'} to the restriction property of chordal $\SLE_{8/3},$ radial $\SLE_{8/3},$ see \cite[Theorem~6.1]{LSW03}, \cite[Theorem~6.26]{Lawler05}, respectively, and to the one-sided restriction property of chordal $\SLE_{8/3}(\frac12(\kappa-6)),$ see \cite[Theorem~8.4]{LSW03}.
The restriction exponents $\lambda$ and $\mu$ can be explained in term of conformal dimensions of $\Psi.$
Let us introduce the \emph{effective boundary condition changing operator} $\Psi^\eff:$
\begin{equation} \label{eq: eff}
\oneleg^\eff := \oneleg \,\PP(q_-)\PP(q_+),
\end{equation}
where $\PP(q_\pm)$ is the ``boundary puncture operator" defined as a $-\frac12b^2$-boundary differential at $q_\pm$ and $\PP(q_\pm) \equiv 1$ in the identity chart of $\mathbb{H}.$
Then
$$\lambda = h(\Psi):=\frac{a^2}2-ab = \frac{6-\kappa}{2\kappa},\qquad \mu=h_\pm(\Psi^\eff):=\frac{a^2}8-\frac{b^2}2 = \frac{(\kappa-2)(6-\kappa)}{16\kappa}.$$

 As an application of the restriction property of dipolar $\SLE_{8/3},$ we prove a dipolar version of Friedrich-Werner's formula.

\begin{theorem}\label{FW}
For distinct $x_j \in \R\sm\{0\} (j=1,\cdots,n)$ and for $b=-\sqrt3/6,$
$$\lim_{\ve\to 0} \ve^{-2n} \mathbb{P} (\SLE_{8/3} \textrm{ hits all slits } [x_j,x_j+i\ve\sqrt2] )=\wh{\E}[T(x_1) \cdots T(x_n) \,\|\,\id_{\mathbb{S}}].
$$
\end{theorem}
Compare Theorem~\ref{FW} to its chordal and radial version (see \cite[Proposition~1]{FW03}  and \cite[Theorem~4.4]{KM12}, respectively).

 In the last subsection we identify the probability that $z(\in D)$ is swallowed by the $\SLE_\kappa (\kappa >4)$ hulls and the probability that $z$ is to the left (right) of the $\SLE_\kappa$ paths with correlations of primary observables (Cardy-Zhan's observables) by the method of ``screening."

\bs\section{Dipolar CFT}

After we discuss central charge modifications of the Gaussian free field in a simply connected domain $D$ with two marked boundary points $q_\pm,$ we define a dipolar version of Ward's  functionals in terms of Lie derivatives.
Based on this approach, we derive Ward's equations in the dipolar case.
For those who are not familiar with some of the definitions and concepts developed in  \cite[Lectures~1, 3~--~5, 7]{KM11}),  we review them in Appendix.

\subsection{Central charge modification} \label{ss: c modify}
For a given simply connected domain $D$ with two marked points $q_\pm \in \pa D$, we consider a conformal transformation
$$w:(D,q_-,q_+) \to (\mathbb{H} ,-1,1)$$
from $D$ onto the half-plane $\mathbb{H}=\{z\in\C\,|\, \Im\,z >0\}.$
For a parameter $b = \sqrt{\kappa/8}-\sqrt{2/\kappa},$ we define the central charge modifications $\Phi\equiv\Phi_{(b)}$ of the Gaussian free field $\Phi_{(0)}$ with the Dirichlet boundary condition by
$$\Phi_{(b)}=\Phi_{(0)} +\varphi, \qquad\varphi=-2b \arg (\frac{w'}{1-w^2}).$$
It is well-defined since $\varphi$ does not depend on the choice of $w.$
We also define the current field $J\equiv J_{(b)}$ by
$$J_{(b)}=\pa\Phi_{(b)}=J_{(0)} +j,\qquad j=ib( \frac{w''}{w'} +\frac{2ww'}{1-w^2}).$$
The current field is an $\Aut(D,q_-,q_+)$-invariant pre-Schwarzian form of order $ib.$
Furthermore, the OPE family $\FF_{(b)}$ is $\Aut(D,q_-,q_+)$-invariant.
Compare the following proposition to \cite[Proposition~10.1]{KM11}.
See Subsections~\ref{ss: stress tensor}~--~\ref{ss: Virasoro field} for the definitions of a stress tensor and the Virasoro field.

\begin{prop}
The bosonic field $\Phi_{(b)}$ has a stress tensor, and the Virasoro field is given by
\begin{equation}
T\equiv T_{(b)}=-\frac{1}{2}J*J +ib\pa J.
\end{equation}
\end{prop}

\begin{proof}
Let us define a holomorphic field $A$ by
\begin{equation} \label{eq: Ab}
A\equiv A_{(b)} =A_{(0)} +(ib\pa -j)J_{(0)}, \qquad A_{(0)} = -\frac12J_{(0)}\odot J_{(0)} .
\end{equation}
Then $A$ is a quadratic differential.
Indeed, as in the chordal and the radial cases, $ib\pa J_{(0)}$ and $jJ_{(0)}$satisfy the following transformation laws:
\begin{align*}
ib\pa J_{(0)} &= ibh'' \tilde J_{(0)}\circ h + ib(h')^2 \pa \tilde J_{(0)}\circ h, \\
jJ_{(0)} &= ib \Big(\frac{h''}{h'}\Big)h' \tilde J_{(0)}\circ h +(h')^2 (\tilde j\tilde J_{(0)}) \circ h.
\end{align*}
Since $\Phi_{(b)}$ is a real part of pre-pre-Schwarzian form of order $ib$, it is enough to check Ward's OPE in the $(\mathbb{H},-1,1)$-uniformization,
$$A(\zeta)\Phi(z) \sim  \frac{ib}{(\zeta -z)^2} +\frac{J_{(b)}}{\zeta -z}.$$
(We use the notation $\sim$ for the singular part of the operator product expansion.)
However, this is immediate from Ward's OPE in the case $b=0$ and the following operator product expansions:
$$\pa J_{(0)}(\zeta)\Phi_{(0)}(z)\sim \frac{1}{(\zeta-z)^2}, \qquad j(\zeta)J_{(0)}(\zeta)\Phi_{(0)}(z)\sim-\frac{j(z)}{\zeta-z}.$$

Finally, let us show that $T_{(b)}$ is the Virasoro field for the OPE family $\FF_{(b)}.$
Since $\FF_{(b)}$ is closed under differentiation and OPE multiplication (see Subsection~\ref{ss: stress tensor} or \cite[Proposition~5.8]{KM11}), $T_{(b)}$ is in $\FF_{(b)}.$ Therefore, it has a stress tensor $(A_{(b)},\overline{A_{(b)}}).$
It follows from the expressions  of $T$ and $J$ that
$$T = A + \frac{1}{12}S_w - \frac{j^2}2 + ibj',$$
where $S_w =(w''/w)'-\frac12(w''/w')^2$ is the Schwarzian derivative of $w.$
The term $-\frac12{j^2} + ibj'$ simplifies $ -b^2 S_w -2b^2 w'^2/(1-w^2)^2.$
Thus we get
\begin{align}
T=A+\frac{1-12b^2}{12}S_w -2b^2 (\frac{w'}{1-w^2})^2
\end{align}
and therefore $T$ is a Schwarzian form of order $\frac1{12}c.$
The central charge $c$ is given by
$$c\equiv 1-12b^2=\frac{(3\kappa-8)(6-\kappa)}{2\kappa}.$$
\end{proof}

\subsection{Ward's functionals and Ward's identities} \label{ss: Ward}
In this subsection we modify the definition of Ward's functionals in the chordal case (see \cite[Sections~5.5~--~5.6]{KM11}) into the dipolar case and derive Ward's identities.
For a given open set $U$ such that $\bar{U} \subset \bar{D} \sm \{q_\pm\}$ and a smooth vector field $v$ on $\bar{U}$, Ward's functional $W^\pm(v;U)$ is defined by
\begin{align*}
W^+(v;U)= \frac{1}{2\pi i}\int_{\pa U}vA - \frac{1}{\pi}\int\!\!\int_U (\bp v)A, \qquad W^-(v;U)=\overline{W^+(v;U)}
\end{align*}
where $A\equiv A_{(b)} =A_{(0)} +(ib\pa -j)J_{(0)},$ $A_{(0)} =-\frac12 J_{(0)}\odot J_{(0)},$ see \eqref{eq: Ab}, and $j=\E[J_{(b)}].$
We also write $W(v;U)=2\,\Re\, W^+(v;U).$
Then $\E\,[W^+(v;U)\,\XX]$ is a well-defined correlation function if $\XX$ is a Fock space functional on $\bar{D} \sm \{q_\pm\}$ with nodes in $U$ and in the maximal open set $D_{\hol}(v)$ where $v$ is holomorphic.

 Recall that the following statements are equivalent (see \cite[Propositions~5.3 and 5.10]{KM11}):
\begin{itemize}
\item Ward's OPE holds for a Fock space field $X.$
\item The residue form of Ward's identity for $X$
$$
\LL_vX(z)= \frac1{2\pi i}\oint_{(z)} vA^+\,X(z)  -\frac1{2\pi i}\oint_{(z)} \bar vA^-\,X(z)
$$
holds on $D_{\hol}(v) \cap U$ for all (local) smooth vector field $v.$
(See Subsection~\ref{ss: Lie} for the definition of Lie derivatives $\LL_v$ and their basic properties.)
\item For all $z \in D_{\hol}(v) \cap U$
\begin{align*}
\E[\YY\LL(v,U)X(z)]=\E [W(v,U)X(z)\YY]
\end{align*} holds for all correlation functional $\YY$ whose nodes are in $(D\sm \bar{U}).$
\end{itemize}

The definition of Ward's functionals can be extended to meromorphic vector fields.
For a meromorphic vector field $v$ which is continuous up to the boundary and has a simple zero at $q_\pm,$
we define Ward's functional $W^+(v;\bar{D}\sm\{q_\pm\})$ by
\begin{align*}
W^+(v;\bar{D}\sm\{q_\pm\})= \lim_{\ve \to 0} W^+(v;D_{\ve}),
\end{align*} where $D_{\ve}=D\sm (B(q_-,\ve)\cup B(q_+,\ve)\cup \bigcup_j B(p_j ,\ve) )$ and $p_j$'s are poles of $v.$
Thus we have
\begin{align*}
W^+(v;\bar{D}\sm\{q_\pm\})=\frac{1}{2 \pi i} \int_{\pa D} vA -\frac{1}{\pi}\int\!\!\int_{D} (\bp v)A,
\end{align*} where $\bp v$ is considered as a distribution.
Note that $vA$ has a removable singularity at $q_\pm.$
Suppose $X_j$'s are in the OPE family $\FF_{(b)}.$
Then the global Ward's identity
\begin{align} \label{eq: global Ward}
\E \,[\LL_vX_1 (z_1 ) \cdots X_n (z_n)]=\E\,[W(v;\bar{D}\sm\{q_\pm\}) X_1(z_1)\cdots X_n (z_n)]
\end{align}
holds if all $z_j \in D_{\hol}(v).$
Compare \eqref{eq: global Ward} to \cite[Proposition~5.9]{KM11}.

We now represent a quadratic differential $A\equiv A_{(b)} =A_{(0)} +(ib\pa -j)J_{(0)}$ in terms of Ward's functionals associated with the dipolar Loewner vector field $v_{\zeta}:$
\begin{align*}
(v_{\zeta}\,\|\,\id_{\overline{\mathbb{H}}\sm\{\pm 1\}})(z)=\frac{1-z^2}{2}\frac{1-\zeta z}{\zeta -z}.
\end{align*}

\begin{prop}\label{A}
In the identity chart $\id_\mathbb{H}$ of $\mathbb{H},$ we have
$$A(\zeta)=\frac{2}{(1-\zeta^2)^2} \Big(W^+ (v_{\zeta};\overline{\mathbb{H}}\sm\{\pm 1\}) +W^- (v_{\bar{\zeta}};\overline{\mathbb{H}}\sm\{\pm 1\})\Big).$$
\end{prop}
\begin{proof}
For $\zeta\in\mathbb{H},$ by the definition of Ward's functional,
$$W^+ (v_{\zeta};\overline{\mathbb{H}}\sm\{\pm 1\}) = \frac{1}{2\pi i} \int_{-\infty}^\infty v_{\zeta} A - \frac{1}{\pi}\int_{\mathbb{H}} (\bp v_{\zeta} )A.$$
On the other hand, the reflected vector field
$$v_\zeta^\#(z) := \overline{v_{\zeta}(\bar z)} = v_{\bar\zeta}(z)$$
is holomorphic in $\mathbb{H}.$
Therefore, we have
$$W^+ (v_{\bar{\zeta}};\overline{\mathbb{H}}\sm\{\pm 1\}) = \frac{1}{2\pi i} \int_{-\infty}^\infty v_{\bar{\zeta}} A.$$
Since $A$ is real on the boundary and $\overline{v_{\bar{\zeta}}}=v_\zeta$ on the boundary,
$$W^- (v_{\bar{\zeta}};\overline{\mathbb{H}}\sm\{\pm 1\}) = -\frac{1}{2\pi i} \int_{-\infty}^\infty v_\zeta A.$$
Proposition now follows from the fact that $\bp v_{\zeta}=-\frac12\pi(1-\zeta^2)^2\delta_\zeta.$
\end{proof}

\subsection{Ward's equations in the upper half-plane}
We will use Ward's equations below (Proposition~\ref{Ward}) to prove Theorem~\ref{main X}.
For $Y\in\FF_{(b)},$ let us express its (holomorphic part of) Lie derivative $\LL_{v_\zeta}^+Y$ in terms of the singular part of operator product expansion
$A(\zeta)Y(z)\sim \sum_{j\le-1}{C_j(z)}{(\zeta-z)^j}$ as $\zeta\to z$ and
the OPE coefficients $ C_j\,(j=-1,-2,-3):$
\begin{align} \label{eq: sing OPE}
\LL_{v_\zeta}^+Y(z) = \sum_{j=-3}^{-1}P_j(\zeta,z)\, C_j(z) +\frac{(1-\zeta^2)^2}{2} \sum_{j\le-1}{C_j(z)}{(\zeta-z)^j}
\end{align}
where $P_{-1}(\zeta,z)=\frac12(2\zeta+z-\zeta^3-\zeta^2z-\zeta z^2),$ $P_{-2}(\zeta,z)=\frac12(1-\zeta^2-2\zeta z),$ $P_{-3}(\zeta,z)=-\frac12\zeta.$
Equation \eqref{eq: sing OPE} can be shown by the identities
$$\LL_{v_\zeta}^+Y(z) = \frac1{2\pi i}\oint_{(z)}v_\zeta(\eta)A(\eta)Y(z)\,\dd \eta = \sum_{j\le-1}{C_j(z)} \frac1{2\pi i}\oint_{(z)}(\eta-z)^j\,v_\zeta(\eta)\,\dd \eta,$$
and
$$\frac1{2\pi i}\oint_{(z)}(\eta-z)^j\,v_\zeta(\eta)\,\dd \eta = \frac{(1-\zeta^2)^2}2(\zeta-z)^j + P_j(\zeta,z)\qquad (j\le-1)
$$
if we set $P_j(\zeta,z)\equiv0$ for $j\le -4.$

 We state Ward's equation in terms of Virasoro generators $L_n.$
Let us recall the definition of $L_n:$
\begin{equation} \label{eq: Ln}
L_n (z):=\frac{1}{2 \pi i} \oint_{(z)} (\zeta -z )^{n+1}T(\zeta) \,\dd\zeta.
\end{equation}
As operators acting on fields, the modes $L_n$ can be viewed as OPE multiplications, $L_nX = T*_{-n-2}X.$
See Subsection~\ref{ss: OPE} for the definition of $*_n.$

\begin{prop}\label{Ward}
For $Y\in\FF_{(b)}$ and  $X = X_1(z_1)\cdots X_n(z_n)\, (X_j\in\FF_{(b)}),$
\begin{align*}
\E \,[Y(z) \LL_{v_z}^+ X] +\E\,[\LL_{v_{\bar{z}}}^- (Y(z)X)]&=\frac{(1-z^2)^2}{2}\E\,[(L_{-2}Y)(z)X] -\frac{3z(1-z^2)}{2}\E\,[(L_{-1}Y)(z)X] \\
&+ \frac{3z^2 -1}{2} \E\,[(L_0 Y)(z)X] +\frac{z}{2}\E\,[(L_1Y)(z)X] +b^2 \E\,[Y(z)X],
\end{align*}
where all fields are evaluated in the identity chart of $\mathbb{H}.$
\end{prop}

\begin{proof}
Subtracting the singular part of OPE and using \eqref{eq: sing OPE}, we have
\begin{align*}
\E[(A*Y)(z)X]&=\lim_{\zeta \to z}\E[A(\zeta)Y(z)X]-\frac{2}{(1-\zeta^2)^2} \E\,[(\LL^+_{v_{\zeta}} Y(z))X]
\\&+ \frac{2}{(1-z^2)^2}\Big(\frac{3z(1-z^2)}{2} \E\,[(A*_{-1}Y)(z)X]\\ &+\frac{1-3z^2}{2}\E\, [(A*_{-2}Y)(z)X] -\frac{z}{2} \E\,[(A*_{-3}Y)(z)X]\Big).
\end{align*}
We now use Proposition~\ref{A} and Leibniz's rule for Lie derivatives to derive
\begin{align*}
\E\,[Y(z) \LL_{v_z}^+ X] &+\E\,[\LL_{v_{\bar{z}}}^- (Y(z)X)] \\ &=\frac{(1-z^2)^2}{2}\E\,[(A*Y)(z)X]-\frac{3z(1-z^2)}{2}\E\,[(A*_{-1}Y)(z)X] \\
&+\frac{3z^2-1}{2} \E\,[(A*_{-2} Y)(z)X] +\frac{z}{2}\E\,[(A*_{-3}Y)(z)X].
\end{align*}
Proposition now follows since $T=A-2b^2/{(1-z^2)^2}$ in the identity chart of $\mathbb{H}.$
\end{proof}

\bs\section{Vertex fields}
In the next section the boundary condition changing operator $\Psi$ will be introduced as a vertex field rooted at two marked boundary points $q_\pm.$
Also we expand our collection of OPE family of $\Phi_{(b)}$ by considering the (rooted) multi-vertex fields with the neutrality condition.
For this purpose, we introduce the formal bosonic fields first and then define the (formal) multi-vertex fields in terms of the formal bosonic fields.
We also use them to describe the relation between the conformal field theory with the Dirichlet boundary condition and one with the Neumann boundary condition.

\subsection{Formal fields} \label{ss: formal}
The formal 1-point fields $\Phiplus_{(0)}$ and $\Phiminus_{(0)}$ can be interpreted as the ``holomorphic part" and the ``anti-holomorphic part" of the Gaussian free field $\Phi_{(0)}$ in the sense that $\Phi_{(0)} = \Phiplus_{(0)}+\Phiminus_{(0)}$ and $\Phiminus_{(0)}=\overline{\Phiplus_{(0)}}.$
By definition they have the following formal correlations
$$\E [\Phiplus_{(0)} (z) \Phiplus_{(0)}(z_0)]=\log \frac{1}{w(z)-w(z_0)},\quad\E[\Phiplus_{(0)}(z) \Phiminus_{(0)}(z_0)]=\log(w(z)-\overline{w(z_0)}),$$
where $w$ is any conformal transformation from $D$ onto $\mathbb{H}.$
Of course, neither $\Phiplus_{(0)}$ nor $\Phiminus_{(0)}$ is a genuine Fock space field.
However, the formal field
$$\sum_{j=1}^n \sigma_j \Phiplus_{(0)}(z_j) -\sigma_{j*}\Phiminus_{(0)}(z_j)$$
is a well-defined (multivalued) Fock space field if and only if the ``neutrality condition"
$$\sum_j (\sigma_j +\sigma_{j*})=0$$
holds.
For example, as a bi-variant field, $\Phiplus_{(0)}(z,z_0) =\Phiplus_{(0)}(z) -\Phiplus_{(0)}(z_0)$ is a multivalued Fock space field:
$$\Phiplus_{(0)}(z,z_0)=\{\Phiplus_{(0)}(\gamma):=\int_\gamma J_{(0)}(\zeta)\,\dd \zeta\,|\, \gamma \textrm{ is a curve from }z_0\textrm{ to } z\}.$$

We now explain why a version of dipolar conformal field theory with the Neumann boundary condition can be developed as the dual of theory with the Dirichlet boundary condition.
First, let us recall the definition of the harmonic conjugate $\wt\Phi_{(b)}$ of bosonic field $\Phi_{(b)}:$
$$\wt \Phi_{(b)}(z,z_0):=2\, \Im\, \Phiplus_{(b)}(z,z_0).$$
We write $\wt \Phi_{(0)}(z):=2\, \Im\, \Phiplus_{(0)}(z)$ so that $\wt \Phi_{(0)}(z,z_0)=\wt \Phi_{(0)}(z)-\wt \Phi_{(0)}(z_0).$
As a formal field of one variable, $\wt\Phi_{(0)}$ has the 2-point correlation function
$$\E\,[\wt\Phi_{(0)}(\zeta)\wt\Phi_{(0)}(z)]=-\E\,[(\Phiplus_{(0)}(\zeta)-\Phiminus_{(0)}(\zeta))(\Phiplus_{(0)}(z)-\Phiminus_{(0)}(z))] =G_N(\zeta,z),$$
where $G_N(\zeta,z)= 2 \log|w(\zeta)-w(z)||w(\zeta)-\overline{w(z)}|$ is the (formal) Green's function of $D$ with the Neumann boundary condition.
Therefore,
$$\E\,[\wt\Phi_{(0)}(\zeta,\zeta_0)\wt\Phi_{(0)}(z,z_0)] = G_N(\zeta,z)-G_N(\zeta_0,z)-G_N(\zeta,z_0)+G_N(\zeta_0,z_0).$$
It is well known that the difference of two Neumann Green's function is well-defined, see \cite{Kenyon01}.
Thus, as a bi-variant Fock space field, a Gaussian free field $\Phi_{(0)}^N(z,z_0)$ with the Neumann boundary condition can be defined in terms of the dual boson $\wt\Phi_{(0)}(z,z_0)$ of the Gaussian free field $\Phi_{(0)}$ with the Dirichlet boundary condition:
$$\Phi_{(0)}^N(z,z_0):= \wt\Phi_{(0)}(z,z_0).$$
For a real parameter $b,$ we define the central charge modification $\Phi_{(b)}^N(z,z_0)$ by
$$\Phi_{(b)}^N(z,z_0) :=\Phi_{(0)}^N(z,z_0) +2b \log \Big|\frac{w'(z)}{1-w(z)^2}\Big|-2b\log \Big|\frac{w'(z_0)}{1-w(z_0)^2}\Big|.$$
Then $\Phi_{(b)}^N(z,z_0) = \wt\Phi_{(b)}(z,z_0).$
Now we define the current $J_{(b)}^N$ and the Virasoro field $T_{(b)}^N$ by
$$J_{(b)}^N(z) = i\pa_z \Phi_{(b)}^N(z,z_0), \qquad T_{(b)}^N=-\frac{1}{2}J_{(b)}^N*J_{(b)}^N +ib\pa J_{(b)}^N$$
so that $J_{(b)}^N = J_{(b)}$ and $T_{(b)}^N = T_{(b)}.$
Also we have $\wt\Phi_{(b)}^N(z,z_0)=-\Phi_{(b)}(z,z_0).$

\subsection{Multi-vertex fields} \label{ss: O}
It is convenient to describe multi-vertex fields in terms of formal fields.
We formally define
$$\OO^{(\sigma)}=M^{(\sigma)}\,e^{\odot i\sigma \Phiplus_{(0)}},\qquad M^{(\sigma)}=\E\,[\OO^{(\sigma)}]=(w')^h (1-w^2)^\mu,$$
where $w$ is any conformal transformation from $(D,q_-,q_+)$ onto $(\mathbb{H},-1,1)$, and $h=\frac12\sigma^2 -\sigma b, \mu = \sigma b.$
As a multivalued Fock space field, a chiral bi-vertex field is defined by
\begin{align}
\OO^{(\sigma)} (z,z_0) \equiv \OO^{(\sigma)}(z) \OO^{(-\sigma)} (z_0) =M^{(\sigma)}(z)M^{(-\sigma)}(z_0)I^{(\sigma)} (z,z_0) e^{\odot i\sigma \Phiplus_{(0)}(z,z_0)},
\end{align}
where the interaction term $I$ is given by
$$I^{(\sigma)}(z,z_0)=e^{\sigma^2 \E[\Phiplus_{(0)}(z)\Phiplus_{(0)}(z_0)]}=(w-w_0)^{-\sigma^2}.$$
Next we define a general 1-point vertex field $\OO^{(\sigma,\sigma_*)}$ by
$$\OO^{(\sigma,\sigma_*)}(z)=\OO^{(\sigma)}(z) \overline{\OO^{(\overline{\sigma_*})}}(z)=M^{(\sigma)} (z)\overline{M^{(\overline{\sigma_*})}}(z) I^{(\sigma,\sigma_*)}(z)e^{\odot i  \sigma\Phiplus_{(0)} (z)-i\sigma_*\Phiminus_{(0)}(z)},$$
where the interaction term $I$ is given by
$$I^{(\sigma,\sigma_*)} (z)=e^{\sigma \sigma_* \E[\Phiplus_{(0)}(z)\Phiminus_{(0)}(z)]}=(w-\bar{w})^{\sigma \sigma_*}.$$
Thus $M^{(\sigma,\sigma_*)}=\E\,[\OO^{(\sigma,\sigma_*)}] = (w')^{h}(\overline{w'})^{h_*}(1-w^2)^{\mu}(1-\bar{w}^2)^{\mu_*}(w-\bar{w})^{\sigma \sigma_*}$ with the dimensions and exponents
$$h=\frac{\sigma^2}{2}-\sigma b ,\quad h_*=\frac{\sigma_*^2}{2}-\sigma_* b,\quad \mu =\sigma b,\quad\mu_*=\sigma_*b.$$
Finally we define the multi-vertex field
\begin{equation}
\OO^{(\bfs{\sigma},\bfs{\sigma_*})}(\bfs{z})=\prod{M^{(\sigma_j,\sigma_{j*})}}(z_j)\prod_{j<k} I_{j,k} (z_j,z_k) e^{\odot i(\sum \sigma_j \Phiplus_{(0)}(z_j) -\sigma_{j*}\Phiminus_{(0)}(z_j))},
\end{equation}
where $\bfs\sigma=(\sigma_1,\cdots,\sigma_n),$ $\bfs{\sigma_*}=(\sigma_{1*},\cdots,\sigma_{n*}),$ and $\bfs{z} = (z_1,\cdots,z_n).$
The interaction terms $I_{j,k} (z_j,z_k)=e^{-\E(\sigma_j \Phiplus_{(0)}(z_j) -\sigma_{j*}\Phiminus_{(0)}(z_j))(\sigma_k \Phiplus_{(0)}(z_k) -\sigma_{k*}\Phiminus_{(0)}(z_k))}$ are given by:
\begin{equation}\label{eq: interaction}
I_{j,k} (z_j,z_k)=(w_j -w_k)^{\sigma_{j} \sigma_{k}}(w_j -\bar{w}_k)^{\sigma_{j} \sigma_{k*}}(\bar{w}_j -w_k)^{\sigma_{j*} \sigma_{k}}(\bar{w}_j -\bar{w}_k)^{\sigma_{j*} \sigma_{k*}}.
\end{equation}
One can view a multi-vertex field as a product of general 1-point vertex fields:
$$\OO^{(\bfs{\sigma},\bfs{\sigma_*})}(\bfs{z})=\OO^{(\sigma_1,\sigma_{1*})}(z_1)\OO^{(\sigma_2,\sigma_{2*})}(z_2)\cdots \OO^{(\sigma_n,\sigma_{n*})}(z_n).$$

\begin{prop}
Under the neutrality condition, $\OO\equiv \OO^{(\bfs{\sigma},\bfs{\sigma_*})}$ are well-defined $\Aut(D,q_\pm)$-invariant Fock space fields.
Moreover, they satisfy Ward's OPE:
$$T(\zeta)\OO(\bfs{z})\sim h_{j\phantom{*}}\frac{\OO(\bfs{z})}{(\zeta-z_j)^2} + \frac{\pa_{z_j}\OO(\bfs{z})}{\zeta-z_j},\qquad
T(\zeta)\bar\OO(\bfs{z})\sim \bar h_{j*}\frac{\bar\OO(\bfs{z})}{(\zeta-z_j)^2} + \frac{\pa_{z_j}\bar\OO(\bfs{z})}{\zeta - z_j},$$
as $\zeta \to z_j.$
\end{prop}

\begin{proof}
Under the neutrality condition, $\sum \sigma_j\Phiplus_{(0)}(z_j) - \sigma_{j*}\Phiminus_{(0)}(z_j)$ is a well-defined Fock space field,
$$\sum \sigma_j \Phi_{(0)}(z_1) + \sum_{j>1} \sigma_j\Phiplus_{(0)}(z_j,z_1) - \sigma_{j*}\Phiminus_{(0)}(z_j,z_1).$$
Suppose $\E\,[\OO^{(\sigma,\sigma_*)}]$ is expressed as $M^{(\bfs{\sigma ,\sigma_*})}, \wt M^{(\bfs{\sigma ,\sigma_*})}$ in terms of two different conformal transformations $w, \wt w$ from $(D,q_-,q_+)$ onto $(\mathbb{H},-1,1).$
Since a nontrivial element $h$ in $\Aut(\mathbb{H},-1,1)$ is of the form
$$h(z)=\frac{az+1}{z+a},\qquad (a\in \R\sm[-1,1]),$$
the ratio of $\wt{M}^{(\bfs\sigma ,\bfs{\sigma_*})}$ and $M^{(\bfs\sigma ,\bfs{\sigma_*})}$ is
$$\frac{(a^2 -1)^{\sum_{j=1}^n(\sigma_j +\sigma_{j*})}}
{\prod(w_j+a)^{\sigma_j^2+\sigma_j\sigma_{j*}+\sum_{k\ne j}\sigma_j(\sigma_k +\sigma_{k*})}(\bar w_j+a)^{\sigma_{j*}^2+\sigma_j\sigma_{j*}+\sum_{k\ne j}\sigma_{j*}(\sigma_k +\sigma_{k*})}}.
$$
Thus the rooted multi-vertex field $\OO^{(\bfs{\sigma},\bfs{\sigma_*})}$ with the neutrality condition is $\Aut(D,q_\pm)$-invariant.

Finally let us show that Ward's OPE holds for $\OO.$
Since the multi-vertex field is a differential, it is enough to verify Ward's OPE in the upper half-plane.
In the identity chart of the upper half-plane,
$$J=J_{(0)} +ib\frac{2z}{1-z^2}, \quad T=T_{(0)}-jJ_{(0)} + ib\partial J_{(0)} + (ib\pa j -\frac12j^2),$$
where $j=2ibz/(1-z^2)$ and $T_{(0)}=-\frac12J_{(0)}\odot J_{(0)}.$
In the simplest case $b=0,$ let us show that the singular part of operator product expansion of $T_{(0)}(\zeta)$ and $\OO(\bfs{z})$ is
$$\frac{\sigma^2}{2} \frac{\OO(\bfs{z})}{(\zeta-z_j)^2} +i\sigma\frac{J_{(0)}(z_j)\odot \OO(\bfs{z})}{\zeta-z_j}
+\Big(\frac{\sigma_j\sigma_{j*}}{z_j-\bar z_j}+\sum_{k\ne j}\frac{\sigma_j\sigma_k}{z_j-z_k}+\frac{\sigma_j\sigma_{k*}}{z_j-\bar z_k}\Big)\frac{\OO(\bfs{z})}{\zeta-z_j}$$
as $\zeta\to z_j.$
For this, let
$$F\equiv F(\zeta,\bfs{z}):=\E[J(\zeta)\sum \big(i\sigma_j\Phiplus_{(0)}(z_j)-i\sigma_{j*}\Phiminus_{(0)}(z_j)\big)] =-i\sum\Big(\frac{\sigma_j}{\zeta-z_j}+\frac{\sigma_{j*}}{\zeta-\bar z_j}\Big).$$
It follows from Wick's calculus that
$$T_{(0)}(\zeta)\,\OO(\bfs{z}) = T_{(0)}(\zeta)\odot \OO(\bfs{z}) - F \,J_{(0)}(\zeta)\odot \OO(\bfs{z}) -\frac12 \,F^2\,\OO(\bfs{z}).
$$
While the first term $T_{(0)}(\zeta)\odot \OO(\bfs{z})$ has no contribution to Ward's OPE for $\OO,$  the second term $- F(\zeta,\bfs{z}) \,J_{(0)}(\zeta)\odot \OO(\bfs{z})$ has the singular part
$$i\sigma\frac{J_{(0)}(z_j)\odot \OO(\bfs{z})}{\zeta-z_j},\qquad (\zeta\to z_j).$$
The singular part of the last term $-\frac12\,F(\zeta,\bfs{z}) ^2\,\OO(\bfs{z})$ is
$$\frac{\sigma^2}{2} \frac{\OO(\bfs{z})}{(\zeta-z_j)^2}
+\Big(\frac{\sigma_j\sigma_{j*}}{z_j-\bar z_j}+\sum_{k\ne j}\frac{\sigma_j\sigma_k}{z_j-z_k}+\frac{\sigma_j\sigma_{k*}}{z_j-\bar z_k}\Big)\frac{\OO(\bfs{z})}{\zeta-z_j},\qquad (\zeta\to z_j).$$
This proves Ward's OPE when $b=0.$
For $b\ne 0,$ we just need to show that
$$j(\zeta)J_{(0)}(\zeta)\,\OO(\bfs{z}) \sim 2\sigma_j \frac{z_j}{1-z_j^2}\frac{\OO(\bfs{z})}{\zeta-z_j},
\quad
ib \pa J_{(0)}(\zeta)\, \OO(z,z_0) \sim  -\sigma_j b\,\frac{\OO(\bfs{z})}{(\zeta-z_j)^2}.$$
Both of the singular OPEs follow from
$$ J_{(0)}(\zeta)\, \OO(z,z_0) \sim  -i\sigma_j \,\frac{\OO(\bfs{z})}{\zeta-z_j}.$$
Ward's OPE for $\bar\OO$ can be obtained in a similar way.
\end{proof}

\subsection{Rooted multi-vertex fields} \label{ss: O*}
In this subsection we introduce the (formal) multi-vertex fields $\OO^{(\bfs{\sigma},\bfs{\sigma_*};\sigma_-,\sigma_+)}$ rooted at two marked boundary points $q_-,q_+.$
The definition of rooted vertex fields can be arrived to by normalizing the tensor product of formal vertex fields, $\OO^{(\bfs\sigma,\bfs{\sigma_*})}(\bfs{z})$, $\OO^{(\sigma_-)}(\zeta_-),$ and $\OO^{(\sigma_+)}(\zeta_+)$ ($\zeta_\pm\in\pa D$) so that the limit exists as $(\zeta_-,\zeta_+)$ tends to $(q_-,q_+).$
For example, we define the rooted vertex field $\OO^{(\sigma,\sigma_*;\sigma_-,\sigma_+)}$ by
$$\OO^{(\sigma,\sigma_*;\sigma_-,\sigma_+)}(z) = M^{(\sigma,\sigma_*;\sigma_-,\sigma_+)}(z) \, e^{\odot i(\sigma\Phiplus_{(0)}(z)-\sigma_*\Phiminus_{(0)}(z)+\sigma_-\Phiplus_{(0)}(q_-)+\sigma_+\Phiplus_{(0)}(q_+))},$$
where $M^{(\sigma,\sigma_*;\sigma_-,\sigma_+)} = \E\,[\OO^{(\sigma,\sigma_*;\sigma_-,\sigma_+)}]$ is given by
\begin{align*}
M^{(\sigma,\sigma_*;\sigma_-,\sigma_+)}(z)&=(1-w)^{\nu^+}(1+w)^{\nu^-} (1-\bar{w})^{\nu_*^+}(1+\bar{w})^{\nu_*^-}(w-\bar{w})^{\sigma \sigma_*}
\\&\times
({w'_-})^{h_-}({w'_+})^{h_+}(w')^{h}(\overline{w'})^{h_*}, \qquad (w'_\pm=w'(q_\pm)).
\end{align*}
The dimensions $[h,h_*;h_-;h_+]$ and exponents are given by
\begin{equation}
h=\frac{\sigma^2}{2}-\sigma b, \quad h_*=\frac{\sigma_*^2}{2}-\sigma_* b , \quad h_\pm=\frac{\sigma_\pm^2}{2},
\end{equation}
and $\nu^\pm=\sigma(b+\sigma_\pm),\,\,
\nu_*^\pm=\sigma_*(b+\sigma_\pm).$
Let us explain this definition.
We express the correlation function $\E\,[\OO^{(\sigma,\sigma_*)}(z)\,\OO^{(\sigma_-)}(\zeta_-)\,\OO^{(\sigma_+)}(\zeta_+)]$ as
$$M^{(\sigma,\sigma_*)}(z)\,M^{(\sigma_-)}(\zeta_-)\,M^{(\sigma_+)}(\zeta_+)\,I_-(z,\zeta_-)\,I_+(z,\zeta_+)\,I(\zeta_-,\zeta_+),$$
where $M^{(\sigma_\pm)}(\zeta_\pm) = \E\,[\OO^{(\sigma_\pm)}(\zeta_\pm)]=(w'(\zeta_\pm))^{\lambda_\pm}(1-w(\zeta_\pm)^2)^{\mu_\pm},$ $(\lambda_\pm = \frac12\sigma_\pm^2-\sigma_\pm b, \, \mu_\pm = \sigma_\pm b)$ and the interaction terms $I_\pm(z,\zeta_\pm), I(\zeta_-,\zeta_+)$ are given by
$$I_\pm(z,\zeta_\pm)=(w-w(\zeta_\pm))^{\sigma\sigma_\pm}(\bar w-w(\zeta_\pm))^{\sigma_*\sigma_\pm}, \quad I(\zeta_-,\zeta_+) = (w(\zeta_-)-w(\zeta_+))^{\sigma_-\sigma_+}.$$
We now apply the following rooting rules to $\E\,[\OO^{(\sigma,\sigma_*)}(z)\,\OO^{(\sigma_-)}(\zeta_-)\,\OO^{(\sigma_+)}(\zeta_+)]:$
\begin{enumerate}
\item the term
$$(1-w(\zeta_\pm)^2)^{\mu_\pm}$$
in $M^{(\sigma_\pm)}(\zeta_\pm)$ is replaced by
$$(w'_\pm)^{\mu_\pm};$$
\item all other terms $w(\zeta_\pm)$ are replaced by $\pm1.$
\end{enumerate}
Thus $h_\pm = \lambda_\pm + \mu_\pm = \frac12\sigma_\pm^2$
and $\nu^\pm=\mu+\sigma\sigma_\pm,\nu_*^\pm=\mu_*+\sigma_*\sigma_\pm.$
The above rooting rules are obtained (up to constant) by
\begin{equation}
\OO^{(\sigma,\sigma_* ; \sigma_-,\sigma_+)}(z):=\lim_{\ve\to 0}
\frac{\OO^{(\sigma,\sigma_*)}(z)\OO^{(\sigma_-)}(\zeta^-_\ve)\OO^{(\sigma_+)}(\zeta^+_\ve)}
{(\zeta^-_\ve-q_-)^{\sigma_-b}(\zeta^+_\ve-q_+)^{\sigma_+b} },
\end{equation}
where $\zeta^\pm_\ve$ is at distance $\ve$ from $q_\pm$ in a given chart.
The definition of rooted 1-point vertex fields can be extended to the rooted multi-vertex fields as follows:
\begin{align*}
\OO^{(\bfs{\sigma},\bfs{\sigma_*};\sigma_-,\sigma_+)}=&(w'_-)^{h_-}(w'_+)^{h_+}\prod M_j \prod_{j<k} I_{j,k} \\&e^{\odot i(\sigma_-\Phiplus_{(0)}(q_-)+\sigma_+\Phiplus_{(0)}(q_+)+\sum \sigma_j \Phiplus_{(0)}(z_j) -\sigma_{j*}\Phiminus_{(0)}(z_j))},
\end{align*}
 where the interaction term $I_{j,k}$ is the same as \eqref{eq: interaction} and
$$M_j=(w'_j)^{h_j}(\overline{w'_j})^{h_{j*}}(1-w_j)^{\nu_j^+}(1+w_j)^{\nu_j^-} (1-\bar{w}_j)^{\nu_{j*}^+}(1+\bar{w}_j)^{\nu_{j*}^-}(w-\bar{w})^{\sigma \sigma_*}$$
with the exponents $\nu_j^\pm=\sigma_j(b+\sigma_\pm),\nu_{j*}^\pm=\sigma_{j*}(b+\sigma_\pm).$
The dimensions $[\bfs{h}, \bfs{h}_*;h_-,h_+]$ of rooted multi-vertex fields $\OO^{(\bfs{\sigma},\bfs{\sigma_*};\sigma_-,\sigma_+)}$ are given by
$$h_j=\frac{\sigma_j^2}{2}-\sigma_j b, \quad h_{j*}=\frac{\sigma_{j*}^2}{2}-\sigma_{j*} b , \quad h_\pm=\frac{\sigma_\pm^2}{2}.$$
If the neutrality condition
$$\sigma_++\sigma_-+\sum(\sigma_j+\sigma_{j*})=0$$
holds, then a rooted vertex field $\OO^{(\bfs{\sigma},\bfs{\sigma_*};\sigma_-,\sigma_+)}$ is a well-defined $\Aut(D,q_-,q_+)$-invariant Fock space field.
While the rooted multi-vertex field $\OO^{(\bfs\sigma,\bfs\sigma_*;\tau,\tau_*)}$ in the radial case can be viewed as the OPE exponential of the field
$$i\tau \Phiplus_{(b)}(q) - i\tau_*\Phiminus_{(b)}(q)+i\sum \sigma_j \Phiplus_{(b)}(z_j)-\sigma_{j*}\Phiminus_{(b)}(z_j),$$
the field $\OO^{(\bfs{\sigma},\bfs{\sigma_*};\sigma_-,\sigma_+)}$ in the dipolar case can be interpreted as the OPE exponential of the field
$$i\sigma_- \Phiplus_{(b)}(q_-) + i\sigma_+\Phiplus_{(b)}(q_+)+i\sum \sigma_j \Phiplus_{(b)}(z_j)-\sigma_{j*}\Phiminus_{(b)}(z_j).$$

Applying the rooting procedure, we define the normalized tensor product of rooted multi-vertex fields by
$$\OO^{(\bfs{\sigma},\bfs{\sigma_*};\sigma_-,\sigma_+)} \star \OO^{(\bfs{\tau,\tau_*};\tau_-,\tau_+)}=\OO^{(\bfs{\sigma+\tau,\sigma_*+\tau_*};\sigma_-+\tau_-,\sigma_++\tau_+)}.$$
One can view $\bfs\sigma,\bfs{\sigma_*}, \bfs\tau,\bfs{\tau_*}$ as divisors, maps from $\bar D\sm\{q_-,q_+\}$ to $\R$ which take the value $0$ at all but finitely many points.

\subsection{Ward's identity and equation for rooted multi-vertex fields}
We now extend the OPE family $\FF_{(b)}$ to include rooted multi-vertex fields.
This extension is natural in the sense that Ward's OPEs for multi-vertex fields survive under the rooting procedure.

\begin{prop}For a rooted multi-vertex field $\OO\equiv \OO^{(\bfs{\sigma},\bfs{\sigma_*};\sigma_-,\sigma_+)},$ Ward's OPE
$$T(\zeta)\OO(\bfs{z}) \sim h_j\frac{\OO(\bfs{z})}{(\zeta -z_j)^2}+\frac{\pa_{z_j}\OO(\bfs{z})}{\zeta-z_j}, \qquad(\zeta \to z_j),$$
holds and similar equation holds (with $\bar h_{j*}$) for $\bar{\OO}.$
\end{prop}
Applying the rooting procedure, we derive the following Ward's identity for  rooted multi-vertex fields.

\begin{prop}
 For a rooted vertex field $\OO(\bfs{z})\equiv \OO^{(\bfs{\sigma},\bfs{\sigma_*};\sigma_-,\sigma_+)}$ with the neutrality condition, we have Ward's identity
\begin{equation}
\E\,[\,W(v;\bar D\sm\{q_\pm\})\,\OO(\bfs{z})\,]=\E\,[\,\LL_v\,\OO(\bfs{z})\,]+(h_-\!+h_+)\,\E[\,\OO(\bfs{z})\,], \, (h_\pm=\frac12\sigma_\pm^2),
\end{equation}
where $v$ is a non-random local holomorphic vector field with $v(q_\pm)=0, v'(q_\pm)=1,$
and $z_j\in D_\hol(v).$
The Lie derivative operator $\LL_v$ does not apply to the points $q_\pm.$
\end{prop}

\begin{proof}
Since vertex fields are differentials, it suffices to perform the computation in the $(\mathbb{H},-1,1)$-uniformization.
Suppose that
$$\OO(\bfs{z})=\lim_{\ve\to0}\frac{\OO^{(\bfs\sigma,\bfs{\sigma_*})}(\bfs{z})\OO^{(\sigma_-)}(\zeta_\ve^-)\OO^{(\sigma_+)}(\zeta_\ve^+)}{(1+\zeta_\ve^-)^{\mu_-}(1-\zeta_\ve^+)^{\mu_+}},$$
where $\zeta_\ve^\pm$ is at distance $\ve$ from $\pm1.$
We write $\OO(\bfs{z},\zeta_-,\zeta_+)$ for the (unrooted) multi-vertex field $\OO^{(\bfs\sigma,\bfs{\sigma_*})}(\bfs{z})\OO^{(\sigma_-)}(\zeta_-)\OO^{(\sigma_+)}(\zeta_+).$

Since Ward's identity holds for $\OO(\bfs{z},\zeta_-,\zeta_+),$ it is enough to show that
$$\lim_{\ve \to 0}\frac{\E\,[\LL_v\, \OO(\bfs{z},\zeta_\ve^-,\zeta_\ve^+)]}{(1+\zeta_\ve^-)^{\mu_-}(1-\zeta_\ve^+)^{\mu_+}}=\E\,[\LL_v \,\OO(\bfs{z})] +(h_-\!+h_+)\,\E\,[\OO(\bfs{z})].$$
Clearly,
$$\lim_{\ve \to 0}\frac{\E\,[\LL_v(\bfs{z})\, \OO(\bfs{z},\zeta_\ve^-,\zeta_\ve^+)]}{(1+\zeta_\ve^-)^{\mu_-}(1-\zeta_\ve^+)^{\mu_+}}=\E\,[\LL_v\, \OO(\bfs{z})],$$
where $\LL_v(\bfs{z})=\sum (v(z_j)\pa_j + \overline{v(z_j)}\bp_j + h_j v'(z_j) + h_{j*} \overline{v'(z_j)}).$
Write $\LL_v(\zeta_\pm)$ for the differential operator, $\sum (v(\zeta_\pm)\pa_{\zeta_\pm} +\lambda_\pm v'(\zeta_\pm)).$
Then
\begin{align*}
\lim_{\ve \to 0}\frac{\E\,[\LL_v(\zeta_\pm)\,\OO(\bfs{z},\zeta_\ve^-,\zeta_\ve^+)]}{(1+\zeta_\ve^-)^{\mu_-}(1-\zeta_\ve^+)^{\mu_+}}
&=\lim_{\ve \to 0}(\mu_\pm \frac{v(\zeta_\ve^\pm)}{\zeta_\ve^\pm\mp1}+\lambda_\pm v'(\zeta_\ve^\pm))\frac{\E\,[\OO(\bfs{z},\zeta_\ve^-,\zeta_\ve^+)]}{(1+\zeta_\ve^-)^{\mu_-}(1-\zeta_\ve^+)^{\mu_+}}\\
&=h_\pm\, \E\,[\OO(\bfs{z})],
\end{align*}
which completes the proof.
\end{proof}

Since the dipolar Loewner vector field
$$v_{\zeta}(z)=\frac{1-z^2}{2}\frac{1-\zeta z}{\zeta -z}$$
in the upper half-plane satisfies $v_{\zeta}(\pm 1)=0, v'_{\zeta}(\pm 1 )=1,$ we can apply the previous proposition to the vector field $v_\zeta$ together with Proposition~\ref{Ward} and derive the following form of Ward's equation for a rooted multi-vertex field.

\begin{prop}\label{TO}
For a rooted multi-vertex field $\OO\equiv \OO^{(\bfs{\sigma},\bfs{\sigma_*};\sigma_-,\sigma_+)}$ in the extended OPE family $\FF_{(b)},$ we have
$$\frac{(1-\zeta^2)^2}{2} \E [T_{(b)} (\zeta) \OO ] =\E [(\LL^+_{v_{\zeta}}+\LL^-_{v_{\bar{\zeta}}})\OO]+(h_-\!+h_+-b^2)\E [\OO],$$
where $T_{(b)}$ and $\OO$ are evaluated in the identity chart of the upper half-plane.
\end{prop}

We now generalize the previous proposition.

\begin{prop}\label{Ward4O}
For a 1-point rooted vertex field $V$, and a rooted multi-vertex field $\OO$ in $\FF_{(b)},$
in the identity chart of the upper half-plane, we have
\begin{align*}
\E\, [V(z)&\star \LL^+_{v_z} \OO] +\E\, [\LL^-_{v_{\bar{z}}} (V(z)\star\OO)]\\
&=\frac{(1-z^2)^2}{2}\,\E\,[(L_{-2}V)(z)\star\OO] -\frac{3z(1-z^2)}{2}\,\E\,[L_{-1}V(z)\star \OO]\\
&+(\frac{3z^2-1}{2}\,h(V) +b^2 -(h_-(V\star \OO)+h_+(V\star \OO)))\,\E\,[V(z)\star \OO],
\end{align*}
where $h(V)$ is the conformal dimension of $V$ with respect to $z$ and $h_\pm(V\star \OO)$ is the boundary dimension of $V\star \OO$ with respect to $q_\pm.$
\end{prop}
\noindent \textit{Remark.}
The Lie derivative operators $\LL_{v}^\pm$ in the above proposition do not apply to the points  $q_\pm.$
For the definition of a field $V\star\LL_v^+\OO,$ the rooting rules can be applied to the tensor product of a bi-vertex field and the Lie derivative of a multi-vertex field.

\begin{proof}
Since Leibniz's rule applies to $\star$-products, i.e.,
$$\LL_{v}^+(V\star\OO) = (\LL_{v}^+V)\star\OO + V\star(\LL_{v}^+\OO),$$
it follows from Proposition~\ref{TO} that
\begin{align*}
\E\,&[V(z)\star \LL^+_{v_\zeta}\OO] + \E\,[\LL^-_{v_{\bar\zeta}} (V(z)\star\OO)] \\
&=\frac{(1-\zeta^2)^2}{2} \,\E\,[T(\zeta) V(z)\star\OO]-\E\,[ (\LL^+_{v_\zeta} V)(z)\star\OO]- (h_-\!+h_+-b^2)\,\E\,[V(z)\star\OO],
\end{align*}
where $h_\pm = h_\pm(V\star \OO).$
By \eqref{eq: sing OPE},
\begin{align*}
\lim_{\zeta\to z} &\frac{(1-\zeta^2)^2}{2} \,T(\zeta) V(z)-\LL^+_{v_\zeta} V(z) \\
&= \frac{(1-z^2)^2}{2}\, T*V(z) -\frac{3z(1-z^2)}{2} \,T*_{-1}V (z)+ \frac{3z^2-1}{2}\,T*_{-2}V(z).
\end{align*}
Since $V$ is a primary field in $\FF_{(b)}$ with conformal dimension $[h,0],$ $L_0V = hV,$ see Proposition~\ref{Virasoro primary}.

\end{proof}

\subsection{Level two degeneracy equations}
In Subsection~\ref{ss: dipolar CFT}  we introduce insertion fields
$$e^{\odot \frac12ia( \Phiplus(p,q_-)+\Phiplus(p,q_+))}\qquad (p\in\pa D\sm \bar Q).$$
These Wick's exponentials can be normalized properly such that the normalized fields form a one-parameter family of $\Aut(D,q_-,q_+)$-invariant (Virasoro) primary fields
$$\Psi:=\OO^{(a,0;-\frac12a,-\frac12a)}$$
in the extended OPE family $\FF_{(b)}.$ By definition of rooted vertex fields,
\begin{equation} \label{eq: Psi}
\Psi (z)= (w'_-)^{\frac18a^2}(w'_+)^{\frac18a^2}\big(\frac{w'(z)}{1-w(z)^2}\big)^h e^{\odot \frac12ia( \Phiplus(z,q_-)+\Phiplus(z,q_+))},
\end{equation}
where $h=\frac12a^2-ab$, and $w$ is a conformal transformation from $(D,q_-,q_+)$ onto $(\mathbb{H},-1,1).$ (Recall that $w'_\pm=w'(q_\pm).$)
As rooted vertex fields, $\Psi$ are current primary fields (see Subsection~\ref{ss: Virasoro field}) with charges $q=a, q_*=0,$ i.e.,
\begin{equation} \label{eq: JPsi}
J_0 \Psi=-ia\Psi,\quad  J_0 \bar\Psi=0, \quad J_n\Psi=J_n\bar\Psi=0 \, (n\ge1),
\end{equation}
where the modes $J_n$ of the current field are defined as
\begin{equation} \label{eq: Jn}
J_n(z) :=\frac{1}{2\pi i}\oint_{(z)}(\zeta -z)^n J(\zeta)\,\dd \zeta.
\end{equation}
To show \eqref{eq: JPsi}, one needs to check the singular OPEs
$$J_{(0)}(\zeta)\Psi(z) \sim -ia\frac{\Psi(z)}{\zeta -z}, \qquad J_{(0)}(\zeta)\overline{\Psi(z)} \sim 0$$
in the identity chart of the upper half-plane.
The first singular OPE follows from Wick's calculus
$$J_{(0)}(\zeta)\Psi(z)=J_{(0)}(\zeta)\odot\Psi(z)+\frac{ia}{2}\,\E[J_{(0)}(\zeta)(\Phiplus_{(0)} (z,-1)+\Phiplus_{(0)}(z,1))]\,\Psi(z)$$
and
$\E[J_{(0)}(\zeta)(\Phiplus_{(0)} (z,-1)+\Phiplus_{(0)}(z,1))]=-2/(\zeta-z) + 2\zeta/(\zeta^2-1).$
The second OPE follows from the similar Wick's decomposition and the fact that
the correlation
$$\E[J_{(0)}(\zeta)(\Phiminus_{(0)} (z,-1)+\Phiminus_{(0)}(z,1))]=-\frac2{\zeta-\bar z} + \frac{2\zeta}{\zeta^2-1}$$
has no singular term.

\begin{prop}\label{level2} If $2a(a+b)=1,$ then
$$T_{(b)} * \Psi =\frac{1}{2a^2} \pa^2 \Psi.$$
\end{prop}

This proposition follows immediately from the characterization of level two degenerate current primary fields, see Proposition~\ref{level2current}.
We combine Proposition~\ref{level2} with Ward's equations to prove Theorems~\ref{main X} and \ref{main O}.

\bs\section{Connection between dipolar SLE and CFT}
After we introduce the insertion fields $\Psi$ as boundary condition changing operators acting on Fock space functionals/fields, we prove that correlation functions of fields in the OPE family $\FF_{(b)}$ of $\Phi_{(b)}$ under the insertion of $\Psi(p)/\E\,[\Psi(p)]$ are dipolar $\SLE_\kappa(p\to Q)$ martingale-observables.
The main ingredient for its proof is BPZ-Cardy equation which is derived from the level two degeneracy equation for $\Psi$ and Ward's equation.
As applications, we discuss the restriction property of dipolar $\SLE_{8/3}$ and the dipolar version of Friedrich-Werner's formula.

\subsection{Boundary condition changing operator}
We define a boundary condition changing operator $\XX \mapsto \wh{\XX}$ as a linear operator acting on Fock space functionals in the following way.
By definition, $\XX \mapsto \wh{\XX}$ is given by the rules
$$\pa \XX \mapsto \pa \wh{\XX} , \quad \bp \XX \mapsto \bp \wh{\XX} , \quad \XX \odot \YY \mapsto \wh{\XX}\odot \wh{\YY},$$
and the formula
\begin{align*}
\sum \sigma_j \Phiplus_{(0)}(z_j) -\sigma_{j*} \Phiminus_{(0)}(z_j) &\mapsto
 \sum-\frac{i\sigma_ja}{2} \log\frac{w_j^2}{1-w_j^2} +\frac{i\sigma_{j*}a}{2} \log\frac{\bar w_j^2}{1-\bar w_j^2}  \\
&+\sum \sigma_j \Phiplus_{(0)}(z_j) -\sigma_{j*} \Phiminus_{(0)}(z_j),
\end{align*}
where $w$ is the conformal transformation from $(D,p,q_-,q_+)$ onto $(\mathbb{H},0,-1,1)$ and the neutrality condition $\sum_j(\sigma_j + \sigma_{j*}) = 0$ holds.

Let us denote by $\wh{\FF}_{(b)}$ the image of $\FF_{(b)}$ under this boundary condition changing operator $\XX \mapsto \wh{\XX}.$
Also we denote
$$\wh{\E}[\XX]:=\frac{\E[\Psi(p)\XX]}{\E[\Psi(p)]}=\E[e^{\odot \frac{1}{2}ia(\Phiplus(p,q_-)+\Phiplus(p,q_+))}\XX].$$
As in the chordal case (\cite[Proposition 14.1]{KM11}), we have
\begin{equation} \label{eq: hat}
\wh{\E}[\XX]=\E[\wh{\XX}],
\end{equation}
where $\XX$ is the string in $\FF_{(b)}$ with nodes in $\bar D\sm\{q_\pm\}.$

\medskip\noindent \textbf{Examples.} Let $w$ be the conformal transformation from $(D,p,q_{\pm})$ onto $(\mathbb{H},0,\pm 1).$

\smallskip \noindent
(a) The bosonic field $\wh\Phi$ is a real part of pre-pre-Schwarzian form of order $ib$,
$$\wh\Phi=\Phi+a \arg\frac{w^2}{1-w^2} =\Phi_{(0)}+a \arg\frac{w^2}{1-w^2} -2b \arg \frac{w'}{1-w^2};$$
(b) The current field $\wh{J}$ is a pre-Schwarzian form of order $ib$,
$$\wh{J}=J-ia \frac{w'}{w(1-w^2)} =J_{(0)}-ia \frac{w'}{w(1-w^2)}+ib(\frac{w''}{w'}+\frac{2ww'}{1-w^2});$$
(c) The Virasoro field $\wh{T}$ is a Schwarzian form of order $\frac1{12}c$,
$$\wh{T}=A_{(0)} - \hat{j}J_{(0)} +ib\pa J_{(0)} +\frac{c}{12} S_w +h_{1,2}\frac{w'^2}{w^2(1-w^2)} +4h_{0,1/2}(\frac{w'}{1-w^2})^2,$$
where $A_{(0)}  = - \frac12J_{(0)} \odot J_{(0)},$ $\hat{j}= \E\,[\wh{J}\,\,]$ and  $h=\frac12a^2-ab,$ $ h_{0,1/2}=\frac18a^2-\frac12b^2;$

\smallskip \noindent
(d) The multi-vertex field $\wh{\OO}^{(\bfs{\sigma},\bfs{\sigma_*})}$ is a $[\bfs{h}, \bfs{h}_*]$-differential ($h_j = \frac12\sigma_j^2 -\sigma_jb,$ $h_{j*} = \frac12\sigma_{j*}^2 -\sigma_{j*}b$),
$$\wh{\OO}^{(\bfs{\sigma},\bfs{\sigma_*})}(\bfs{z})=\prod \wh{M}^{(\sigma_j,\sigma_{j*})}(z_j)\, \prod_{j<k} I_{j,k} \, e^{i\odot \sum \sigma_j \Phiplus_{(0)}(z_j) -\sigma_{j*} \Phiminus_{(0)}(z_j)},$$
where $\wh{M}^{(\sigma_j,\sigma_{j*})}(z_j)=(w'_j)^{h_j}(\overline{w'}_j)^{h_{j*}}
(1-w^2_j)^{\wh{\mu}_j}(1-\bar{w}^2_j)^{\wh{\mu}_{j*}}w_j^{\sigma_ja}\bar{w}_j^{\sigma_{j*}a}(w_j-\bar w_j)^{\sigma \sigma_*}$ and interaction term $I_{j,k}$ is the same as \eqref{eq: interaction}.
The exponents are given by
$$\wh{\mu}_j=\mu_j -\frac12\sigma_ja=\sigma_j \big(b-\frac12a\big) ,\quad \wh{\mu}_{j*}=\mu_{j*}-\frac12\sigma_{j*}a=\sigma_{j*} \big(b-\frac12a\big).$$

\subsection{BPZ-Cardy equations}
We now derive BPZ-Cardy equations in the dipolar case.
Suppose $X=X_1 (z_1)\cdots X_n (z_n)$ is the tensor product of fields $X_j$ in the OPE family $\FF_{(b)}.$
For $\xi \in \R$, we denote
$$\wh{\E}_{\xi}[X]=\E[e^{\odot \frac12ia(\Phiplus_{(0)}(\xi,-1)+\Phiplus_{(0)}(\xi,1))}X].$$

\begin{prop}\label{Cardy}
If $2a(a+b)=1,$ then we have
\begin{equation}
\wh{\E}_{\xi}[\LL_{v_{\xi}}X]=\frac{1}{2a^2}\Big(\frac{(1-\xi^2)^2}{2}\pa^2_{\xi} -\xi(1-\xi^2)\pa_{\xi}\Big)\wh{\E}_{\xi} [X], \quad v_{\xi}(z):= \frac{1-z^2}2\frac{1-\xi z}{\xi-z},
\end{equation}
where all fields are evaluated in the identity chart of $\mathbb{H}$ and $\pa_\xi = \pa + \bp.$
\end{prop}
\begin{proof}
In the $(\mathbb{H},0,-1,1)$-uniformization, the rooted vertex field $\Psi=\OO^{(a,0;-\frac12a,-\frac12a)}$ is evaluated at $\xi$ as
$$\Psi(\xi)=(1-\xi^2)^{-h}e^{\odot \frac12ia(\Phiplus_{(0)}(\xi,-1)+\Phiplus_{(0)}(\xi,1))},$$
where $h=\frac12a^2 -ab.$
For $\zeta \in \mathbb{H},$ let
$$R_{\zeta} \equiv R(\zeta; z_1 ,z_2 , \cdots ,z_n)=\E[(1-\zeta^2)^h \Psi(\zeta) X].$$
By Ward's equation (Proposition~\ref{Ward}), $L_{-1}\Psi = \pa \Psi,$ and level two degeneracy equation (Proposition~\ref{level2}) for the rooted vertex field $\Psi,$ we have
\begin{align*}
\E[\Psi(\zeta)(\LL^+_{v_{\zeta}}X+\LL^-_{v_{\bar{\zeta}}}X)]&=\frac1{2a^2}\frac{(1-\zeta^2)^2}{2}\E[(\partial^2\Psi)(\zeta)X]-\frac{3\zeta(1-\zeta^2)}{2}\E[(\pa \Psi)(\zeta)X]\\
&+(\frac{3\zeta^2-1}{2}h +b^2 -\frac{a^2}{4})\E[\Psi(\zeta)X].
\end{align*}
(We also use the fact that $\Psi$ is a holomorphic field, and therefore $\LL_v^-\Psi(\zeta)=0.$)
Due to the numerology $2a(a+b)=1,$ it simplifies that
$$\E[(1-\zeta^2)^h \Psi(\zeta)(\LL^+_{v_{\zeta}}X+\LL^-_{v_{\bar{\zeta}}}X)] =\frac{1}{2a^2}\Big(\frac{(1-\zeta^2)^2}{2}\pa^2 R_{\zeta} -\zeta (1-\zeta^2)\pa R_{\zeta}\Big),$$
where $\pa$ is the operator of differentiation with respect to the complex variable $\zeta.$
We now take the limits of both sides as $\zeta\to\xi.$
Since $\xi$ is real, the left-hand side converges to
$$\E[(1-\xi^2)^h \Psi(\xi)\LL_{v_{\xi}}X]=\wh{\E}_{\xi}[\LL_{v_{\xi}}X].$$
On the other hand, since $\pa_\xi = \pa + \bp,$ and the rooted vertex field $\Psi$ is holomorphic, $\pa R_\zeta,$ and $\pa^2 R_\zeta$ in the right-hand sides converge to $\pa_\xi R_\xi$ and $\pa_\xi^2 R_\xi,$ respectively.
\end{proof}

\subsection{Dipolar SLE martingale-observables} \label{ss: MOs}
It is convenient to describe dipolar SLEs in terms of the $(\mathbb{H},-1,1)$-uniformization.
Let $\xi_t =(e^{\sqrt{\kappa}B_t}-1)/(e^{\sqrt{\kappa}B_t}+1)$ and let $g_t$ be the dipolar SLE map from $(D_t ,\gamma_t,Q)$ onto $(\mathbb{H},\xi_t,\R\sm[-1,1]).$
Then $g_t$ satisfies
\begin{equation} \label{eq: dgt}
\pa_t g_t (z)=-\frac{1-g_t^2(z)}{2}\frac{1-\xi_t g_t(z)}{\xi_t-g_t(z)},
\end{equation}
where $g_0:(D, p,Q)\to(\mathbb{H},0, \R\sm[-1,1])$ is the conformal map from $D$ onto the upper half-plane $\mathbb{H}.$
Let us restate Theorem~\ref{main X} and present its proof.

\begin{theorem}If $X_j$'s are Fock space fields in the OPE family $\FF_{(b)}$, then a non-random field
$$M(z_1, \cdots, z_n ) =\wh{\E} [ X_1 (z_1) \cdots X_n (z_n)]$$
is a martingale-observable for dipolar $\SLE_{\kappa}.$
\end{theorem}

\begin{proof}
Conformal invariance allows us to represent the process
$$M_t (z_1,\cdots,z_n)\equiv M_{(D_t ,\gamma_t , Q)}(z_1,\cdots, z_n)$$
as
$$M_t =m(\xi_t ,t) ,\,\, m(\xi, t )=(R_{\xi} \,\|\, g^{-1}_t ),$$
where $g_t :(D_t,\gamma_t , q_-,q_+) \to (\mathbb{H},\xi_t ,-1,1)$ is the dipolar SLE map driven by $\xi_t$ and
$$R_{\xi} (z_1,\cdots,z_n)=\wh{\E}_{\xi}[X_1(z_1)\cdots X_n(z_n)].$$
It\^o's formula can be applied to $m(\xi_t ,t)$ since the function $m(\xi,t)$ is smooth in both $\xi$ and $t.$
Since the driving process $\xi_t =(e^{\sqrt{\kappa}B_t}-1)/(e^{\sqrt{\kappa}B_t}+1)$ satisfies
$$\dd\xi_t =\frac{\sqrt{\kappa}}{2}(1-\xi_t^2)\,\dd B_t -\frac{\kappa}{4}\,\xi_t(1-\xi_t^2)\,\dd t,$$
It\^o's formula  shows that  $M_t$ is a semi-martingale with the drift term
$$\frac{\kappa}{4}\Big(\frac{(1-\xi^2)^2}{2}\pa_{\xi}^2 -\xi (1-\xi^2)\pa_{\xi}\Big)\Big|_{\xi=\xi_t}m(\xi ,t)\,\dd t +\frac{\dd}{\dd s}\Big|_{s=0}(R_{\xi}\,\|\, g_{t+s}^{-1})\,\dd t,$$
where $L_t = \pa_s\big|_{s=0}(R_{\xi}\,\|\, g_{t+s}^{-1})=\pa_s\big|_{s=0} (R_{\xi}\,\|\, g_t^{-1} \circ f_{s,t}^{-1}),$ and $f_{s,t} =g_{t+s} \circ g_t^{-1}.$
It follows from \eqref{eq: dgt} that the time-dependent flow $f_{s,t}$ satisfies
$$\frac{\dd}{\dd s}f_{s,t} (\zeta) = - v_{\xi_{t+s}}(f_{s,t}(z)), \quad v_{\xi}(z):= \frac{1-z^2}2\frac{1-\xi z}{\xi-z}.$$
Thus $f_{s,t}$ can be approximated by $\id -s v_{\xi_t}+o(s)$ as $s\to 0.$
Since $X_j$'s depend smoothly on local charts,
$$L_t =-(\LL_{v_{\xi_t}}R_{\xi_t}\,\|\,g_t^{-1}).$$
It follows from BPZ-Cardy equations that $M_t$ is driftless.
\end{proof}

\subsection{The restriction property}
In this subsection we present CFT theoretic proof for the restriction property of the dipolar $\SLE_{8/3};$ the dipolar $\SLE_{8/3}$ path in $(\mathbb{H},-1,1)$ conditioned to avoid a fixed compact hull $K$ with $ \pa K \cap \R \subseteq (-1,1)\sm\{0\}$ has the same distribution as the dipolar  $\SLE_{8/3}$ path in $(\mathbb{H}\sm K,-1,1).$

 Let $\kappa \leq 4.$
On the event $\gamma(0,\infty)\cap K =\emptyset$, we denote $\Omega_t =g_t(D_t\sm K),\wt{\gamma}=\psi_K \circ \gamma$  and define a conformal map $h_t:\Omega_t\to\mathbb{H}$ by
$$h_t:=\wt{g}_t\circ \psi_K \circ g^{-1}_t,$$
where $\wt{g}_t$ is a dipolar Loewner map from $(\mathbb{H} \sm \wt{\gamma}[0,t],-1,1)$ onto $(\mathbb{H},-1,1)$ and $\psi_K$ is the conformal transformation from $(\mathbb{H} \sm K ,-1,1)$ onto $(\mathbb{H},-1,1)$ such that $\psi'_K(-1)=\psi'_K(1).$
Let
$$M_t=(1-\xi_t^2)^{\lambda}\E(\Psi_{\Omega_t}^{\eff}(\xi_t) \,\|\, \id_{\Omega_t}),$$
where $\Psi^\eff$ is the effective boundary condition changing operator, see \eqref{eq: eff}.
Then
\begin{equation} \label{eq: M}
M_t=\Big(\frac{(1-\xi_t^2)h_t'(\xi_t)}{1-h_t(\xi_t)^2}\Big)^{\lambda}(h_t'(-1)h_t'(1))^{\mu},
\end{equation}
where
\begin{align*}
\lambda=h(\Psi^{\eff})=\frac{6-\kappa}{2\kappa}, \quad \mu=h_-(\Psi^{\eff})=h_+(\Psi^{\eff})=\frac{a^2}{8}-\frac{b^2}{2}=\frac{(\kappa-2)(6-\kappa)}{16\kappa}.
\end{align*}

\begin{lem}\label{restriction1} The process $M_t$ is a semi-martingale with the drift term
$$\frac{c}{24}(1-\xi_t^2)^2\,S_{h_t}(\xi_t) \,M_t \,\dd t.$$
\end{lem}

\begin{proof}
Let
$F(z,t):=\E(\Psi_{\Omega_t}^{\eff}(z) \,\|\,\id_{\Omega_t}).$
Then $M_t=(1-\xi_t^2)^{\lambda}F(\xi_t,t).$
Application of It\^o's formula to the smooth function $F$ gives the drift term of ${\dd M_t}/{M_t}:$
$$\frac{\dot{F}(\xi_t,t)}{F(\xi_t,t)}-\frac{\kappa}{4}(1+2\lambda)(1-\xi_t^2)\xi_t\frac{F'(\xi_t,t)}{F(\xi_t,t)}+\frac{\kappa}{8}(1-\xi_t^2)^2\frac{F''(\xi_t,t)}{F(\xi_t,t)}+\frac{\kappa}{2}(\lambda^2 +\frac{\lambda}{2})\xi_t^2 -\frac{\kappa}{4}\lambda.$$
We represent
$$\dot{F}(z,t)=\frac{\dd}{\dd s}\Big|_{s=0}\E(\Psi^{\eff}_{\mathbb{H}}\,\|\,h^{-1}_{t+s})(z)=\frac{\dd}{\dd s}\Big|_{s=0}\E(\Psi^{\eff}_{\mathbb{H}}\,\|\,h^{-1}_{t}\circ f_{s,t}^{-1})(z), \quad(f_{s,t}=h_{t+s} \circ h_t^{-1})$$
in terms of Lie derivative
\begin{equation} \label{eq: Fdot}
\dot{F}(z,t)=\E(\LL(v,\overline{\mathbb{H}})\Psi_{\mathbb{H}}^{\eff}\,\|\,h_t^{-1})(z),\qquad
(v\,\|\,\id_{\mathbb{H}})=\frac{\dd}{\dd s}\Big|_{s=0} f_{s,t}=\dot{h}_t \circ h_t^{-1},
\end{equation}
where the Lie derivative operator $\LL(v,\overline{\mathbb{H}})$ applies to the points $\pm1.$
To compute the vector field $v$, we apply the chain rule to $h_t=\wt{g}_t\circ \psi_K \circ g_t^{-1}$ and compute the capacity changes.
Indeed,
$$\dot{h}_t (z)=-\Big(\frac{(1-\xi_t^2)h'_t(\xi_t)}{1-h_t(\xi_t)^2}\Big)^2 v_{h_t(\xi_t)}(h_t(z))+h'_t (z)v_{\xi_t}(z), \quad  v_{\xi}(z):= \frac{1-z^2}2\frac{1-\xi z}{\xi-z}.$$
Thus
\begin{equation} \label{eq: v}
(v\,\|\,\id_{\mathbb{H}})(\zeta)=-\Big(\frac{(1-\xi_t^2)h'_t(\xi_t)}{1-h_t(\xi_t)^2}\Big)^2  v_{h(\xi_t)}(\zeta)+h'_t (h_t^{-1}(\zeta))\,v_{\xi_t}(h_t^{-1}(\zeta)).
\end{equation}
By \eqref{eq: Fdot} and \eqref{eq: v}, we have
\begin{align*}
\dot{F}(z,t)&=-\Big(\frac{(1-\xi_t^2)h'_t(\xi_t)}{1-h_t(\xi_t)^2}\Big)^2 h'_t(z)^{\lambda}\,\E(\LL(v_{h_t(\xi_t)} ,\overline{\mathbb{H}})\Psi_\mathbb{H}^{\eff}(h_t (z))\,\|\,\id_{\mathbb{H}})\\
&+\E(\LL(v_{\xi_t},\Omega_t\sm\{\pm 1\})\Psi_{\Omega_t}^\eff\,\|\,\id_{\Omega_t})(z)+2\mu\,\E(\Psi_{\Omega_t}^\eff\,\|\,\id_{\Omega_t})(z),
\end{align*}
where the Lie derivative operator $\LL(v_{\xi_t},\Omega_t\sm\{\pm 1\})$ does not apply to the points $\pm1.$
It follows from Ward's equation that
\begin{align*}
\dot{F}(z,t)&=-\frac{(1-\xi_t^2)^2}{2}h'_t(\xi_t)^2h'_t(z)^{\lambda}\E(T_{\mathbb{H}}(h_t(\xi_t))\Psi_\mathbb{H}^\eff(h_t(z))\,\|\,\id_{\mathbb{H}})\\&+\E(\LL(v_{\xi_t},\Omega_t\sm\{\pm 1\})\Psi_{\Omega_t}^\eff\,\|\,\id_{\Omega_t})(z)+2\mu\,\E(\Psi_{\Omega_t}^\eff\,\|\,\id_{\Omega_t})(z).
\end{align*}
By conformal invariance,
\begin{align*}
\dot{F}(z,t)&=-\frac{(1-\xi_t^2)^2}{2}\E(T_{\Omega_t}(\xi_t)\Psi_{\Omega_t}^\eff(z)\,\|\,\id_{\Omega_t})+\frac{c}{24}(1-\xi_t^2)^2S_{h_t}(\xi_t)\E(\Psi_{\Omega_t}^\eff\,\|\,\id_{\Omega_t})(z)\\
&+\E(\LL(v_{\xi_t},\Omega_t\sm\{\pm 1\})\Psi_{\Omega_t}^\eff\,\|\,\id_{\Omega_t})(z)+2\mu\,\E(\Psi_{\Omega_t}^\eff\,\|\,\id_{\Omega_t})(z).
\end{align*}
Using \eqref{eq: sing OPE} and the fact that $L_{-1}\Psi=\pa\Psi, L_0\Psi=\lambda\Psi, L_1\Psi=0,$ we have
\begin{align*}
\frac{\dot{F}(\xi_t,t)}{F(\xi_t,t)}&=-\frac{(1-\xi_t^2)^2}{2}\frac{(\E\,T_{\Omega_t}*\oneleg^\eff_{\Omega_t}(\xi_t)\,\|\,\id)}{(\E\,\oneleg^\eff_{\Omega_t}(\xi_t)\,\|\,\id)}+\frac{3\xi_t (1-\xi_t^2)}{2}\frac{F'(\xi_t,t)}{F(\xi_t,t)}\\&-(\frac{3\xi_t^2-1}{2}\lambda-2\mu)+\frac{c}{24}(1-\xi_t^2)^2 S_{h_t}(\xi_t).
\end{align*}
Plugging the above equation into the drift term of $\dd M_t,$ lemma now follows from the level two degeneracy equation for $\Psi^\eff.$
\end{proof}

 From now on, we use the $(\mathbb{S},-\infty,\infty)$-uniformization.
By abuse of notation, let $\xi_t=\sqrt\kappa B_t$ (cf. $\xi_t$ in Subsection~\ref{ss: MOs}) and let $g_t$ be the dipolar SLE map from $(D_t ,\gamma_t,Q)$ onto $(\mathbb{S},\xi_t,\R+\pi i).$
As before, for a compact hull $K$ with $K \cap (\R + \pi i) = \emptyset,$  we denote $\Omega_t =g_t(D_t\sm K),\wt{\gamma}=\psi_K \circ \gamma$ and define a conformal map $h_t:\Omega_t\to\mathbb{S}$ by
$$h_t:=\wt{g}_t\circ \psi_K \circ g^{-1}_t,$$
where $\wt{g}_t$ is a dipolar Loewner map from $(\mathbb{S} \sm \wt{\gamma}[0,t],-\infty,\infty)$ onto $(\mathbb{S},-\infty,\infty)$ and $\psi_K$ is the conformal transformation from $(\mathbb{S} \sm K ,-\infty,\infty)$ onto $(\mathbb{S},-\infty,\infty)$ such that
$$\lim_{z\to\pm\infty} \psi_K(z) - z = \pm \,\, \scap(K).$$
Then the process~\eqref{eq: M} in the $(\mathbb{H},-1,1)$-uniformization becomes
$$M_t = (h_t'(\xi_t))^\lambda\,e^{-2\mu\,\scap(K_t)}, \quad K_t = \overline{\mathbb{H}\sm\Omega_t}$$
in the $(\mathbb{S},-\infty,\infty)$-uniformization.

\begin{proof}[Proof of Theorem~\ref{restriction}]
By Lemma~\ref{restriction1}, the process $M_t = (h_t'(\xi_t))^\lambda\,e^{-2\mu\,\scap(K_t)}$ is a local martingale if $\kappa =8/3.$
We first claim that the process $M_t$ is a bounded continuous martingale.
Since $\scap(K)\ge0$ for a compact hull $K$ with $K \cap (\R + \pi i) = \emptyset,$ it suffices to show that $(0<) \psi'_K(x) \le 1$ for $x\in\R\sm K.$
As in the chordal case, $\Im\, \psi_K(z) - \Im\, z$ is a bounded harmonic function with non-positive boundary values.
Thus $\Im\, \psi_K(z) \le\Im\, z$ for $z\in \mathbb{S}\sm K$ and $\psi'_K(x) \le 1$ for $x\in\R\sm K.$

 Let $T=\inf\{t\ge 0\,:\, \gamma(0,t]\cap K \ne \emptyset \}.$
It follows from the martingale convergence theorem that $\lim_{t\to T} M_t$ exists a.s.
The proof that $\lim_{t\to T} M_t = 1_{T=\infty}$ a.s. is similar to that in the chordal case (see \cite[Theorem~6.1]{LSW03}).
We leave it as an exercise for the reader.
By the optional stopping theorem,
$$\psi'_K(0)^\lambda e^{-2\mu\,\scap(K)}=M_0 = \E\,M_T = \mathbb{P}\{T=\infty\}.$$
This proves the theorem.
\end{proof}

 We now prove Friedrich-Werner's formula in the dipolar case.

\begin{proof}[Proof of Theorem~\ref{FW}]
Denote $\bfs{x} = (x_1,\cdots,x_n),$ and
$$R(\xi;\bfs{x}) = \E\,[\,\Psi^\eff(\xi)\,T(x_1)\cdots T(x_n)\,].$$
The non-random field $R\equiv R(\xi;\bfs{x})$ has the following properties:

\smallskip\quad(R1) it is a boundary differential of dimension $\lambda=5/8$ with respect to $\xi;$

\smallskip\quad(R2) it is a boundary differential of dimension $2$ with respect to $x_j;$

\smallskip\quad(R3) it is a boundary differential of dimension $\mu= 5/96$ with respect to $q_\pm.$

\smallskip We apply Ward's equations to the function $R(\xi;\bfs{x})$ so that we replace $T(x)$ in $R(\xi;x,\bfs{x})$ by the Lie derivative operator:
\begin{equation} \label{eq: recursion4T}
R(\xi;x,\bfs{x})=(\LL({v_x},\bar{\mathbb{S}})+\mu)\,R(\xi;\bfs{x}),\quad (\textrm{in }\id_{\bar{\mathbb{S}}}),
\end{equation}
where $v_x(\zeta) = \frac12\coth_2(x-\zeta)$ and the Lie derivative operator $\LL({v_x},\bar{\mathbb{S}})$ do not apply to the points $\pm\infty.$

 Let
$$U(\bfs{x}) = \lim_{\ve\to 0} \ve^{-2n} \mathbb{P} (\SLE_{8/3} \textrm{ hits all slits } [x_j,x_j+i\ve\sqrt2] )$$
(if the limit exists).
We define the non-random field $T(\xi;\bfs{x})$ as follows:
{\setlength{\leftmargini}{1.7em}
\begin{itemize}
 \item it satisfies the transformation laws (R1)~--~(R3);
 \item $(T(\xi;x_1,\cdots,x_n)\,\|\,\id_{\bar{\mathbb{S}}}) = U(x_1-\xi,\cdots,x_n-\xi).$
\end{itemize}}

 We now claim that the limit $U(x,\bfs{x})$ exists under the assumption of existence of the limit $U(\bfs{x})$ and that
\begin{equation} \label{eq: FW0}
T(0;x,\bfs{x}) = (\LL({v_x},\bar{\mathbb{S}}) + \mu)\,T(0;\bfs{x}).
\end{equation}
The non-random fields $T(0;\cdot)$ and $R(0;\cdot)$ satisfy the same recursive equation  (see \eqref{eq: recursion4T} and \eqref{eq: FW0}).
Thus $T(0;\cdot) = R(0;\cdot) $ for all $n\ge 1$ since $T(0;\cdot) = R(0;\cdot) = 1$ for $n=0.$
Therefore, we have $U(\bfs{x}) = R(0;\bfs{x}).$

To verify this claim, we write $\mathbb{P}(\bfs{x})$ for the probability that dipolar $\SLE_{8/3}$ path hits all segments $[x_j,x_j+i\ve\sqrt2]$ and $\mathbb{P}(\bfs{x}\,|\,\neg x)$ for the same probability conditioned on the event that the path avoids $[x,x+i\ve\sqrt2].$
By the induction hypothesis,
\begin{equation} \label{eq: FW1}
\mathbb{P}(\bfs{x})\approx \ve^{2n} \, T(0;\bfs{x}).
\end{equation}
On the other hand, it follows from the restriction property of dipolar $\SLE_{8/3}$ that
\begin{equation} \label{eq: FW2}
1-\mathbb{P}(x) =(\psi'(0))^\lambda \, (e^{-\scap([x,x+i\ve\sqrt2])})^{2\mu}
\end{equation}
and
\begin{equation} \label{eq: FW3}
\mathbb{P}(\bfs{x}\,|\,\neg x)\approx \ve^{2n} \, T(\psi(0);\psi(x_1),\cdots, \psi(x_n))\prod_{j=1}^n\psi'(x_j)^2 , \end{equation}
where $\psi$ is a slit map from $(\mathbb{S}\sm [x,x+i\ve\sqrt2],0,\pm\infty)$ onto $(\mathbb{S},0,\pm\infty).$
This $\psi$ satisfies
$$\psi(z) = \varphi_t(z-x) - \varphi_t(-x), \quad \cosh_2 \varphi_t(z)= e^{\frac12t}\cosh_2z, \quad e^t = 1+ \tan^2(\ve/\sqrt2),$$
where $\cosh_2 z = \cosh(\frac12z).$
By \eqref{eq: FW1}~--~\eqref{eq: FW3},
we approximate $\ve^{-2n}\mathbb{P}(x,\bfs{x})$ by
$$T(0;\bfs{x})-\psi'(0)^\lambda(e^{-\scap([x,x+i\ve\sqrt2])})^{2\mu} \,T(\psi(0);\psi(x_1),\cdots, \psi(x_n))\prod_{j=1}^n\psi'(x_j)^2.$$
Thus the limit $U(x,\bfs{x})$ exists.
Since $\scap([x,x+i\ve\sqrt2]) = \log(1+\tan_2^2\sqrt2\ve)\approx \frac12\ve^2,$ we have
$$T(0;x,\bfs{x}) = (\LL(v_x,\bar{\mathbb{S}})+\mu)\,T(0;\bfs{x}).$$
\end{proof}

\bs\section{Vertex observables}
We expand our OPE family $\FF_{(b)}$ of $\Phi_{(b)}$ by considering the rooted multi-vertex fields with the neutrality condition.
In this section we extend Theorem~\ref{main X} to this expanded family.
As examples of screening of rooted vertex observables, we discuss Cardy-Zhan's observables
that describe the probability for a point to be to the left (right) of the dipolar SLE path and the probability for a point to be swallowed by the dipolar SLE hull.

\subsection{Rooted Multi-vertex fields}
We apply the rooting rules (in Subsection~\ref{ss: O*}) to the multi-vertex fields $\wh\OO^{(\bfs{\sigma},\bfs{\sigma_*})}\wh{\OO}^{(\sigma_-)}\wh{\OO}^{(\sigma_+)}$ and arrive to the definition of rooted multi-vertex fields $\wh{\OO}^{(\bfs{\sigma},\bfs{\sigma_*};\sigma_-,\sigma_+)}:$
$$\wh{\OO}^{(\bfs{\sigma},\bfs{\sigma_*};\sigma_-,\sigma_+)}= \wh{M}^{(\bfs{\sigma},\bfs{\sigma_*};\sigma_-,\sigma_+)}\,e^{\odot i(\sigma_-\Phiplus_{(0)}(q_-)+\sigma_+\Phiplus_{(0)}(q_+)+\sum \sigma_j \Phiplus{(0)}(z_j)-\sigma_{j*}\Phiminus_{(0)}(z_j))},$$
where $\wh{M}^{(\bfs{\sigma},\bfs{\sigma_*};\sigma_-,\sigma_+)}=\E[\wh{\OO}^{(\bfs{\sigma},\bfs{\sigma_*};\sigma_-,\sigma_+)}]=(w'_-)^{\wh{h}_-}(w'_+)^{\wh{h}_+}\prod \wh{M}_j\prod_{j<k}I_{j,k}.$ The interaction term $I_{j,k}$ is the same as \eqref{eq: interaction} and
\begin{align*}
\wh{M}_j=(w'_j)^{h_j}(\overline{w'_j})^{h_{j*}}w_j^{\sigma_ja} \bar w_j^{\sigma_{j*}a} (w_j-\bar{w}_j)^{\sigma_j \sigma_{j*}}(1-w_j)^{\wh{\nu}_+}(1+w_j)^{\wh{\nu}_-}(1-\bar{w}_j)^{\wh{\nu}_{+*}}(1+\bar{w}_j)^{\wh{\nu}_{-*}}
\end{align*}
with the exponents
$\wh{\nu}_\pm=\sigma_j(b-\frac12a+\sigma_\pm),\wh{\nu}_{\pm *}=\sigma_{j*}(b-\frac12a+\sigma_{\pm *}).$
The dimensions $[\bfs{h}, \bfs{h}_*;\wh h_-,\wh h_+]$ of rooted multi-vertex fields $\wh\OO^{(\bfs{\sigma},\bfs{\sigma_*};\sigma_-,\sigma_+)}$ are given by
$$h_j=\frac{\sigma_j^2}{2}-\sigma_j b,\qquad h_{j*}=\frac{\sigma_{j*}^2}{2}-\sigma_{j*} b, \qquad \wh{h}_\pm=\frac{\sigma_\pm^2}{2}-\frac{\sigma_\pm a}{2}.$$

Alternatively, one can define $\wh{\OO}^{(\bfs{\sigma},\bfs{\sigma_*};\sigma_-,\sigma_+)}$ by the action of boundary condition changing operator $\XX\mapsto\wh\XX$ on  $\OO^{(\bfs{\sigma},\bfs{\sigma_*};\sigma_-,\sigma_+)}.$
Indeed, the boundary condition changing operator $\XX\mapsto\wh\XX$ can be extended to formal fields/functionals by the formula
$$\Phiplus_{(0)} \mapsto \Phiplus_{(0)}  -\frac{ia}{2}\log\frac{w^2}{1-w^2}, \qquad
\Phiplus_{(0)}(q_\pm) \mapsto \Phiplus_{(0)}(q_\pm)+\frac{ia}{2}\log w'_\pm$$
and the property that it commutes with complex conjugation.
Note that the interaction terms are preserved under the boundary condition changing operator.
For the rooted multi-vertex field $\OO\equiv\OO^{(\bfs{\sigma},\bfs{\sigma_*};\sigma_-,\sigma_+)},$
$$\wh\OO =
(w'_-)^{-\frac12\sigma_-a}(w'_+)^{-\frac12\sigma_+a}\prod w_j^{\sigma_j a}(1-w_j^2)^{-\frac12\sigma_j a}\bar{w}_j^{\sigma_{j*}a}(1-\bar{w}_j^2)^{-\frac12\sigma_{j*}a}\,\OO.$$
Thus two definitions coincide.

For rooted multi vertex fields $\OO\equiv\OO^{(\bfs{\sigma},\bfs{\sigma_*};\sigma_-,\sigma_+)}$ with the neutrality condition, let us denote
$$
\wh{\E}\, [\OO]:=\frac{\E[\Psi(p)\star\OO]}{\E\Psi (p)},\qquad \wh{\E}_\zeta[\OO]:=\frac{\E[\Psi(\zeta)\star\OO]}{\E\Psi (\zeta)}, \quad(\zeta \in \bar{D}\sm\{q_\pm\}).$$
Then Equation~\eqref{eq: hat} can be extended to the rooted multi-vertex fields:
$$\wh{\E}\,[\OO]=\E\,[\wh{\OO}].$$

Proposition~\ref{Cardy} (BPZ-Cardy equations) extends to the rooted multi-vertex fields.

\begin{prop}\label{Cardy4O}
Suppose that the parameters $a$ and $b$ are related as $2a(a+b)=1.$ Then for rooted multi-vertex fields $\OO\equiv \OO^{(\bfs{\sigma},\bfs{\sigma_*};\sigma_-,\sigma_+)}$ with the neutrality condition,
\begin{equation}
\wh{\E}_{\xi}[\LL_{v_{\xi}}\OO]+(\wh{h}_-\!+\wh{h}_+)\wh{\E}_{\xi}[\OO]=\frac{1}{2a^2}\Big(\frac{(1-\xi^2)^2}{2}\pa^2_{\xi} -\xi(1-\xi^2)\pa_{\xi}\Big)\wh{\E}_{\xi} [\OO],
\end{equation}
where all fields are evaluated in the identity chart of the upper half-plane and $\pa_\xi = \pa + \bp.$ The vector field $v_\xi$ is given by
$$v_{\xi}(z):= \frac{1-z^2}2\frac{1-\xi z}{\xi-z}.$$
\end{prop}

\begin{proof}
For $\zeta \in \overline{\mathbb{H}}\sm \{\pm 1\},$ let us denote
$$R_{\zeta}\equiv R(\zeta,z_1, \cdots ,z_n)=\E[(1-\zeta^2)^h \Psi(\zeta)\star \OO],$$
where $h=\frac12a^2-ab.$
Then it follows from Ward's equation (Proposition~\ref{Ward4O}), $L_{-1}\Psi=\pa\Psi,$ and the level two degeneracy equation for $\Psi$ (Proposition~\ref{level2}) that
\begin{align*}
\E[\Psi \star (\LL^+_{v_\zeta}\OO +\LL^-_{v_{\overline{\zeta}}} \OO)]&=\frac1{2a^2}\frac{(1-\zeta^2)^2}{2}\E[(\pa^2\Psi)\star\OO]-\frac{3\zeta(1-\zeta^2)}{2}\E[(\pa \Psi)\star\OO]\\&+(\frac{3\zeta^2-1}{2}h+b^2-h_-(\Psi\star\OO)-h_+(\Psi\star\OO))\E[\Psi \star \OO],
\end{align*}
where $h_\pm(\Psi\star\OO)$ is the dimension of boundary differential $\Psi\star\OO$ with respect to $q_\pm.$
(We also use the holomorphicity of $\Psi,$ and therefore $\LL_v^-\Psi(\zeta)=0.$)
By the numerology $2a(a+b)=1$ and the relation $h_\pm(\Psi \star \OO)=h_\pm(\Psi)+h_\pm(\OO)=\frac18a^2+\wh{h}_\pm,$ we have
\begin{align*}
\E[(1-\zeta^2)^h\Psi &\star (\LL^+_{v_\zeta}\OO +\LL^-_{v_{\overline{\zeta}}} \OO)]+(\wh{h}_-\!+\wh{h}_+)\E[(1-\zeta^2)^h\Psi \star \OO]\\&=\frac{1}{2a^2}\Big(\frac{(1-\zeta^2)^2}{2}\pa^2-\zeta(1-\zeta^2)\pa\Big)R_\zeta,
\end{align*}
where $\pa$ is the operator of differentiation with respect to the complex variable $\zeta.$
Taking the limits of both sides as $\zeta \to \xi,$ we obtain BPZ-Cardy equations.
\end{proof}

Now we prove Theorem~\ref{main O}.
\begin{proof}[Proof of Theorem~\ref{main O}]
Denote
$$R_{\xi}=\wh{\E}_{\xi}[\OO^{(\bfs{\sigma},\bfs{\sigma_*};\sigma_-,\sigma_+)}].$$
By conformal invariance, the process $M_t \equiv M_{(D_t, \gamma_t, Q)}$ is represented by
$$M_t =m(\xi_t ,t),\,\,m(\xi, t)=(R_{\xi} \,\|\, g_t^{-1}),$$
where $g_t$ is the dipolar SLE map from $(D_t,\gamma_t,q_\pm)$ onto $(\mathbb{H},\xi_t,\pm1).$
Since $g'_t(q_\pm)=e^{-t}$, the drift term of $\dd M_t$ is equal to
$$\frac{1}{2a^2}\Big(\frac{(1-\xi^2)^2}{2}\pa^2_{\xi}-\xi(1-\xi^2)\pa_{\xi}\Big)\Big|_{\xi=\xi_t}m(\xi ,t)\dd t
-(\LL_{v_{\xi_t}}R_t \,\|\,g_t^{-1})\dd t -(\wh{h}_-\!+\wh{h}_+)M_t \dd t,$$
where the Lie derivative operator $\LL_{v_{\xi_t}}$ does not apply the marked boundary points $q_\pm.$
It follows from Proposition~\ref{Cardy4O} (BPZ-Cardy equations) that $\dd M_t$ is driftless.
\end{proof}

\subsection{Cardy-Zhan's observables}
Let $\kappa>4, z\in D.$
We consider the following geometric observables
$$N(z)=\mathbb{P}(\tau_z <\infty), \quad \mathbb{P}(z \textrm { is to the left of }\gamma), \quad \mathbb{P}(z \textrm { is to the right of }\gamma).$$
They are real-valued with all conformal dimensions zero.
There is no such vertex observable except for the constant field.
However, the derivative $\pa N$ can be identified as a vertex field with conformal dimensions
$$\lambda_z =1, \quad \lambda_{z*}=\lambda_{q+}=\lambda_{q-}=0.$$
Indeed, by dimension calculus, a vertex field $\OO^{(-2a,0;a;a)}$ satisfies the above requirements.
Let $M$ be a martingale observables with all conformal dimensions zero such that
$$\pa M(z) = \E [\OO^{(-2a,0;a;a)}] =  w'  \sinh^{-4/\kappa}(\frac w2)\quad\textrm{up to multiplicative constant},$$
where $w$ is the conformal map from $(D,p,q_\pm)$ onto $(\mathbb{S},0,\pm\infty).$
In $\mathbb{S},$ let us choose $M$ satisfying the normalization $M(-\infty)=0,\,M(\infty)=1,$ and $ \Im\, M(0)>0.$
It follows from Schwarz-Christoffel formula that $M$ is the conformal transformation from $(\mathbb{S},0,\pm \infty)$ onto the isosceles triangle with angles $2\pi/\kappa$ at $M(-\infty)=0,\,M(\infty)=1.$
Applying the optional stopping theorem and using the fact that
$$
\begin{cases}
M_{\tau_z}=M(0) &\textrm{ if }\tau_z <\infty, \\
M_{\tau_z}=M(-\infty)=0 \, &\textrm{ if } z\textrm{ is to the left of } \gamma,\\
M_{\tau_z}=M(\infty)=1 \, &\textrm{ if } z\textrm{ is to the right of } \gamma,
\end{cases}
$$
we justify Cardy-Zhan's formulas (see \cite[Corollary~2.3.1]{Zhan04})
\begin{align*}
\mathbb{P}(\tau_z <\infty)=\frac{\Im\, M(z)}{\Im\, M(0)}, \quad \mathbb{P}(z\textrm{ is to the right of } \gamma)=\Re\, M(z)-\frac{1}{2}\frac{\Im\, M(z)}{\Im\, M(0)}.
\end{align*}
\noindent \textit{Remark.}
If $z = x+ \pi i\in\R + \pi i,$ then Cardy-Zhan's observables
$$M(x+\pi i) = \frac{~\displaystyle\int_{-\infty}^x \cosh^{-4/\kappa}(\frac\xi2)\,\dd \xi~}{~\displaystyle\int_{-\infty}^\infty \cosh^{-4/\kappa}(\frac\xi2)\,\dd \xi~}$$
give the distribution of the endpoint $\gamma(\infty)$ of dipolar $\SLE_\kappa\, (\kappa>0)$ path.

\renewcommand\thesection{Appendix.}
\section{Basic properties of conformal Fock space fields} 
\setcounter{section}{7}
\renewcommand\thesection{A}
\setcounter{subsection}{0}
\setcounter{equation}{0}
\setcounter{theorem}{0}
Here, we include some definitions and concepts of conformal Fock space fields developed in \cite{KM11} so that this article can be read as a self-contained one.

\subsection{Fock space correlation functionals}
This subsection is borrowed from \cite[Section~1.3]{KM11}.
By definition, \emph{basic} correlation functionals are formal expressions of the type (Wick's product of $X_j(z_j)$)
$$\XX = X_1(z_1)\odot\cdots\odot X_n(z_n),$$
where points $z_j\in D$ are not necessarily distinct and $X_j$'s are derivatives of the Gaussian free field, (i.e., $X_j=\pa^j\bp^k\Phi$), or Wick's exponentials
$$e^{\odot \alpha\Phi}=\sum_{n=0}^\infty \frac{\alpha^n}{n!}\Phi^{\odot n}.$$
The constant $1$ is also included to the list of basic functionals.
We write $S_\XX$ for the set of all points $z_j$ (the \emph{nodes} of $\XX$) in the expression of $\XX.$

For derivatives $X_{jk}$ of the Gaussian free field and basic functionals of the form
$$\XX_j = X_{j1}(z_{j1}) \odot \cdots \odot X_{jn_j}(z_{jn_j}),$$
we define the tensor product $\XX_1\cdots \XX_m$ by
\begin{equation} \label{eq: Wick's rule}
\XX_1\cdots \XX_m = \sum\prod_{\{v,v'\}} \E[X_{v}(z_{v})X_{v'}(z_{v'})] \underset{v''}{\textstyle\bigodot} X_{v''}(z_{v''}),
\end{equation}
where the sum is taken over Feynman diagrams with vertices $v$ labeled by functionals $X_{jk}$ such that there are no contractions of vertices with the same $j,$ and the Wick's product is taken over unpaired vertices $v''.$
By definition, $\E[X_{v}(z_{v})X_{v'}(z_{v'})]$ in \eqref{eq: Wick's rule} are given by the 2-point functions of derivatives of the Gaussian free field, e.g.,
$$\E[\pa^j\Phi(\zeta)\pa^k\Phi(z)] = \pa_\zeta^j \pa_z^k\E[\Phi(\zeta)\Phi(z)] = 2 \pa_\zeta^j \pa_z^k G(\zeta,z).$$
For example, the Feynman diagram with two edges $\{1,4\},\{3,5\}$ and two unpaired vertices $2,6$
corresponds to
\begin{gather*}
\contraction{}{(\Phi}{(z_1) \odot \Phi(z_2) \odot \Phi(z_3))(}{\Phi} \contraction[2ex]{(\Phi(z_1) \odot \Phi(z_2) \odot }{\Phi}{(z_3))(\Phi(z_4) \odot}{\Phi} (\Phi(z_1) \odot \Phi(z_2) \odot \Phi(z_3))(\Phi(z_4) \odot \Phi(z_5) \odot \Phi(z_6))\\
 := \E[\Phi(z_1)\Phi(z_4)] \E[\Phi(z_3)\Phi(z_5)] \Phi(z_2)\odot \Phi(z_6)
\end{gather*}
The definition of tensor product can be extended to general correlation functionals by linearity.
The tensor product of correlation functionals is commutative and associative, see \cite[Proposition~1.1]{KM11}.

We define the correlation of $\E[\XX]$ of $\XX$ by linearity, $\E[1] = 1,$ and
$$\E[X_1(z_1)\odot\cdots\odot X_n(z_n)] = 0,$$
where $X_j$ are derivatives of  $\Phi.$
For example, $\E[e^{\odot\alpha\Phi(z)}] = 1$ and
$$\E[\Phi(z_1)\cdots \Phi(z_n)] = \sum\prod_k 2G(z_{i_k},z_{j_k}),$$
where the sum is over all partitions of the set $\{1,\cdots,n\}$ into disjoint pairs $\{i_k,j_k\}.$

If $\E[\XX_1\YY] =\E[\XX_2\YY]$ holds for all $\YY$ with nodes outside $S_{\XX_1}\cup S_{\XX_2},$ we identify $\XX_1$ with  $\XX_2$ and write $\XX_1\approx \XX_2.$
We consider Fock space functionals modulo an ideal $\NN=\{\XX\approx0\}$ of Wick's algebra.
The concept of a correlation functional $\XX$ can be extended to the case when some of the nodes of $\XX$ lie on the boundary.
For example, $e^{\odot\alpha\Phi(z)} = 1$ for $z\in\pa D.$
The complex conjugation $\overline\XX$ of $\XX$ is defined (modulo $\NN$) by the equation
$\E[\overline\XX\YY]=\overline{\E[\XX\YY]}$
for all $\YY$'s of the form $\Phi(z_1)\odot\cdots\odot\Phi(z_n).$
For example, if $J=\pa\Phi$ in the half-plane $\mathbb{H}$ and if $z\in\pa\mathbb{H},$ then
$J(z)$ is purely imaginary, i.e., $\overline{J(z)}=-J(z),$ and $J(z)\odot J(z)$ is real.

\subsection{Fock space fields}  \label{ss: F-fields}
This subsection is borrowed from \cite[Section~1.4]{KM11}.
Basic Fock space fields $X_\alpha$ are formal expressions written as Wick's products of derivatives of the Gaussian free field $\Phi$ and Wick's exponential $e^{\odot\alpha\Phi},$ e.g., $1,\, \Phi\odot\Phi\odot \Phi,\, \pa^2 \Phi\odot \bp\Phi,\, \pa\Phi\odot e^{\odot\alpha\Phi},$ etc.
A general Fock space field is a linear combination of basic fields $X_\alpha,$
$$X=\sum_\alpha f_\alpha X_\alpha,$$
where $f_\alpha$'s are arbitrary (smooth) functions in $D.$
If $X_1,\cdots,X_n$ are Fock space fields and $z_1,\cdots,z_n$ are distinct points in $D,$ then
$\XX = X_1(z_1)\cdots X_n(z_n)$
is a correlation functional.

We define the differential operators $\pa$ and $\bp$ on Fock space fields by specifying their action on basic fields so that the action on $\Phi$ is consistent with the definition of $\pa\Phi, \bp\Phi$ and so that
$$\pa (X\odot Y)=(\pa X)\odot Y+ X\odot (\pa Y),\qquad \bp (X\odot Y)=(\bp X)\odot Y+ X\odot (\bp Y).$$
We extend this action to general Fock space fields by linearity and by Leibniz's rule with respect to multiplication by smooth functions.
Then (modulo $\NN$)
$$\E[(\pa X)(z)\YY] = \pa_z \E[X(z)\YY],\qquad (z\not\in S_\YY),$$
for all correlation functionals $\YY.$

By definition, $X$ is \emph{holomorphic} in $D$ if $\bp X\approx 0,$ i.e., all correlation functions $\E[X(\zeta) \YY]$ are holomorphic in $\zeta\in D\sm S_\YY.$
For example, $J=\pa \Phi,$ $X = J\odot J$ are holomorphic fields.

\subsection{Operator product expansion} \label{ss: OPE}
This subsection is borrowed from \cite[Sections~3.1~--~3.2]{KM11}.
Operator product expansion (OPE) is the expansion of the tensor product of two fields near diagonal.
For example,
\begin{equation} \label{eq: OPE(Phi,Phi)}
\Phi(\zeta) \Phi(z) = \log \frac{1}{|\zeta-z|^2} + 2c(z) + \Phi^{\odot 2}(z) + o(1) \qquad\textrm{as }\;\zeta\to z,\;\zeta\ne z,
\end{equation}
where $c=\log C$ is the logarithm of conformal radius $C,$ i.e., $c(z) = u(z,z),$  $u(\zeta,z) = G(\zeta,z) + \log|\zeta-z|.$
The meaning of the convergence is the following: the equation
$$\E[\Phi(\zeta) \Phi(z)\XX] = \log \frac{1}{|\zeta-z|^2}\E[\XX] + 2c(z)\E[\XX] + \E[\Phi^{\odot 2}(z)\XX] + o(1)$$
holds for all Fock space correlation functionals $\XX$ in $D$ satisfying $z\notin S_\XX.$
To derive \eqref{eq: OPE(Phi,Phi)} we use Wick's formula
\eqref{eq: Wick's rule},
$$\Phi(\zeta) \Phi(z) = \E[\Phi(\zeta) \Phi(z)] + \Phi(\zeta)\odot\Phi(z)$$
and the relation
$$\E[\Phi(\zeta) \Phi(z)] = 2G(\zeta,z) = \log \frac{1}{|\zeta-z|^2} + 2c(z) + o(1).$$
The convergence of $\Phi(\zeta)\odot\Phi(z)$ to $\Phi^{\odot2}(z)$ means (by definition) that
$$\E[(\Phi(\zeta)\odot\Phi(z))\XX]\to\E[\Phi^{\odot2}(z)\XX]$$
for every $\XX$ such that $z\not\in S_\XX.$

If the field $X$ is \emph{holomorphic} (i.e., all correlation functions $\E[X(\zeta) \YY]$ are holomorphic in $\zeta\in D\sm S_\YY$), then the operator product expansion is then defined as a (formal) Laurent series expansion
\begin{equation} \label{eq: OPE(X,Y)}
X(\zeta)Y(z)= \sum {C_n(z)}{(\zeta-z)^n}, \qquad\zeta\to z.
\end{equation}
Since the function $\zeta\mapsto\E\,X(\zeta)Y(z)\ZZ$ is holomorphic in a punctured neighborhood of $z,$ it has a Laurent series expansion.

There are only finitely many terms in the principle (or \emph{singular}) part of the Laurent series \eqref{eq: OPE(X,Y)}.
We use the notation $\sim$ for the singular part of the operator product expansion,
$$X(\zeta)Y(z)\sim\sum_{n<0} {C_n(z)}{(\zeta-z)^n}.$$
We also write $\Sing_{\zeta\to z}\,X(\zeta)Y(z)$ for $\sum_{n<0} {C_n(z)}{(\zeta-z)^n}.$
It is clear that we can differentiate operator product expansions \eqref{eq: OPE(X,Y)} both in $\zeta$ and $z$; and the differentiation preserves singular parts.
For example,
$$J(\zeta)\Phi(z)\sim -\frac1{\zeta-z}, \qquad J(\zeta)J(z)\sim-\dfrac1{(\zeta-z)^2}.$$
The coefficients in the operator product expansions (e.g., $2c(z) + \Phi^{\odot2}(z)$ in \eqref{eq: OPE(Phi,Phi)}, $C_n(z)$ in \eqref{eq: OPE(X,Y)}) are called \emph{OPE coefficients}.
OPE coefficients of Fock space fields are Fock space fields (as functions of $z$).
In particular, if $X$ is \emph{holomorphic}, then we define the $*_n$ product by
$X *_{n} Y=C_n.$
We write $*$ for $*_{0}$ and call $X*Y$ the \emph{OPE multiplication}, or the \emph{OPE product} of $X$ and $Y.$

Vertex fields are defined as OPE-exponentials of $\Phi:$
$$\VV^\alpha=e^{*\alpha\Phi}=\sum_{n=0}^\infty\frac{\alpha^n}{n!}\Phi^{*n}.$$
They can be expressed in terms of Wick's calculus: $\VV^\alpha=C^{\alpha^2}e^{\odot\alpha\Phi},$ see \cite[Proposition~3.3]{KM11}.
Here, $C=e^c$ is the conformal radius, see \eqref{eq: OPE(Phi,Phi)}.
The Virasoro field is defined as OPE square of $J = \pa\Phi:$
$$T=-\frac12 J* J.$$
Then by Wick's calculus,
$$T=-\frac12~J\odot J+\frac1{12}S,$$
where $S(z) = S(z,z), S(\zeta,z) := -12\partial_\zeta\partial_z u(\zeta,z), $ and $u(\zeta,z) = G(\zeta,z) + \log|\zeta-z|.$
Thus $T$ is a Schwarzian form of order $\frac1{12}.$
In terms of a conformal map $w:D\to\mathbb{H},$ $S = S_w,$ the Schwarzian derivative of $w.$

\subsection{Conformal Fock space fields} \label{ss: Lie}
This subsection is borrowed from \cite[Sections~4.2~--~4.4]{KM11}.
We use Lie derivative of a conformal field to define the stress tensor and to state Ward's identities.

A general conformal Fock space field is a linear combination of basic fields $X_\alpha,$
$$X=\sum_\alpha f_\alpha X_\alpha,$$
where $f_\alpha$'s are non-random conformal fields, see Subsection~\ref{ss: SLE MO}.
A non-random conformal field $f$ is said to be \emph{invariant} with respect to some conformal automorphism $\tau$ of $M$ if
$$ (f\,\|\,\phi)=(f\,\|\,\phi\circ \tau^{-1})$$
for all charts $\phi.$
For example, suppose $D$ is a planar domain and let us write $f$ for $(f\,\|\,\id_D).$
Then $f$ is a $\tau$-invariant $[\lambda,\lambda_*]$-differential if
$$f(z) =f(\tau z)~\tau'(z)^\lambda~\overline{\tau'(z)}^{\lambda_*}.$$
It is because $\tau$ is the transition map between the charts $\phi\circ \tau^{-1}$ and $\phi=\id_D.$
By definition, a \emph{random} conformal field (or a family of conformal fields) is $\tau$-\emph{invariant} if all correlations are invariant as non-random conformal fields.

Suppose a non-random smooth vector field $v$ is holomorphic in some open set $U\subset M.$
For a conformal Fock space field $X,$ we define the Lie derivative $\LL_vX$ in $U$ as
$$(\LL_v X\,\|\, \phi) = \frac{\dd}{\dd t}\Big|_{t=0} (X\,\|\, \phi\circ\psi_{-t}),$$
where $\psi_t$ is a local flow of $v,$ and $\phi$ is an arbitrary chart.

Lie derivative of a differential is a differential but Lie derivative of a Schwarzian form is a quadratic differential:
\begin{itemize}
\item $\LL_vX=\left(v\pa+\bar v\bar\pa+\lambda v'+\lambda_*\overline{v'}\right)X$ for a $[\lambda,\lambda_*]$-differential $X;$
\item\ms $\LL_vX=\left(v\pa+v'\right)X +\mu v'$ for a pre-Schwarzian form $X$ of order $\mu;$
\item\ms $\LL_vX=\left(v\pa+2v'\right)X +\mu v'''$ for a Schwarzian form $X$ of order $\mu.$
\end{itemize}
We recall basic properties of Lie derivatives:
\begin{itemize}
\item $\LL_v$ is an $\R$-linear operator on Fock space fields;
\item\ms $\E[\LL_vX]=\LL_v(\E[X]);$
\item\ms $\LL_v (\bar X)=\overline {(\LL_vX)}$;
\item\ms $\LL_v(\pa X)=\pa(\LL_v X)$ and $\LL_v(\bar\pa X)=\bar\pa(\LL_v X);$
\item\ms Leibniz's rule applies to Wick's products, OPE products, and tensor products.
\end{itemize}
We define the $\C$-linear part $\LL_v^+$ and anti-linear part $\LL_v^-$ of the Lie derivative $\LL_v$ by
$$2\LL_v^+ = \LL_v-i\LL_{iv},\qquad 2\LL_v^- = \LL_v+i\LL_{iv}.$$

\subsection{Stress tensor} \label{ss: stress tensor}
This subsection is borrowed from \cite[Sections~5.2~--~5.3]{KM11}.
A Fock space field $X$ in $D$ is said to have a (symmetric) stress tensor $(A,\bar A)$ ($X\in\FF(A,\bar A)$) if $A$ is a holomorphic quadratic differential
and if  Ward's OPE holds for $X,$ i.e., on a given chart $\phi:U\to\phi U,$
$$\Sing_{\zeta\to z}[A(\zeta)X(z)]= (\LL_{k_\zeta}^+X)(z), \quad \Sing_{\zeta\to z}[A(\zeta)\bar X(z)]= (\LL_{k_\zeta}^+\bar X)(z),$$
where the (local) vector field $k_\zeta$ is defined by $(k_\zeta\,\|\,\phi)(\eta) = 1/(\zeta-\eta).$
Ward's family $\FF(A,\bar A)$ is closed under differentiation and OPE multiplication, see \cite[Proposition~5.8]{KM11}.
In the case of differentials or forms, it is enough to verify Ward's OPEs in just one chart.
For example, a $[\lambda, \lambda_*]$-differential $X$ is in  $\FF(A,\bar A)$ if and only if the following operator product expansions hold in every/some chart:
$$
A(\zeta)X(z)\sim\frac {\lambda X(z)}{(\zeta-z)^2}+\frac {\pa X(z)}{\zeta-z},\quad A(\zeta)\bar X(z)\sim \frac {\bar \lambda_*\bar X(z)}{(\zeta-z)^2}+\frac {\pa \bar X(z)}{\zeta-z}.
$$
Let $X$ be a form of order $\mu.$
Then $X\in \FF(A,\bar A)$ if and only if the following operator product expansion holds in every/some chart:
\begin{align*}
A(\zeta)X(z)&\sim \frac{\mu}{(\zeta-z)^2}+\frac {\pa X(z)}{\zeta-z} &\textrm{ for a pre-pre-Schwarzian form }X;\\
A(\zeta)X(z)&\sim \frac{2\mu}{(\zeta-z)^3}+\frac {X(z)}{(\zeta-z)^2}+\frac {\pa X(z)}{\zeta-z} &\textrm{ for a pre-Schwarzian form }X;\\
A(\zeta)X(z)&\sim \frac{6\mu}{(\zeta-z)^4}+ \frac {2X(z)}{(\zeta-z)^2}+\frac {\pa X(z)}{\zeta-z} &\textrm{ for a Schwarzian form }X.
\end{align*}
For example, Gaussian free field $\Phi$ has a stress tensor
$$A=-\frac12 J\odot J,\qquad J = \pa\Phi.$$
This holomorphic quadratic differential $A$ coincides with the Virasoro field $T$ in the upper half-plane uniformization.
While $A$ itself does not belong to $\FF(A,\bar A),$ the Virasoro field $T$ is in $\FF(A,\bar A).$
We review the abstract theory of Virasoro field in the next subsection.

\subsection{Virasoro field} \label{ss: Virasoro field}
This subsection is borrowed from \cite[Lecture~7 and Appendix~11]{KM11}.
A Fock space field $T$ is said to be the \emph{Virasoro field} for Ward's family $\FF(A,\bar A)$ if
\begin{itemize}
\item $T\in\FF(A,\bar A),$ and
\smallskip\item $T-A$ is a non-random holomorphic Schwarzian form.
\end{itemize}
We define Virasoro primary fields and current primary fields in terms of Virasoro generators $L_n$ \eqref{eq: Ln} and current generator $J_n$ \eqref{eq: Jn}.

\begin{prop}[Proposition~7.5 in \cite{KM11}] \label{Virasoro primary}
Let $X$ be a Fock space field. Any two of the following assertions imply the third one (but neither one implies the other two):
\begin{itemize}
\smallskip \item \label{item: primary1} $X\in\FF(A,\bar A)$;
\smallskip \item \label{item: primary2} $X$ is a $[\lambda,\lambda_*]$-differential;
\smallskip \item \label{item: primary3} $L_{\ge1}X=0,\; L_0X=\lambda X, \; L_{-1}X=\partial X,$ and similar equations hold for $\bar X.$
\end{itemize}
\end{prop}
Here, $L_{\ge k}X = 0$ means that $L_nX = 0$ for all $n\ge k.$
We call fields satisfying all three conditions \emph{(Virasoro) primary} fields in $\FF(A,\bar A).$

A (Virasoro) primary field $X$ is called \emph{current primary} if
$$J_{\ge1}X = J_{\ge1}\bar X = 0,$$
and
$$J_0X = -iqX, \quad J_0\bar X = i\bar q_*\bar X$$
for some numbers $q$ and $q_*.$
They are called ``charges" of $X$.
Here, $J_{\ge k}X = 0$ means that $J_nX = 0$ for all $n\ge k.$
We use the following proposition to prove Proposition~\ref{level2} (level two degeneracy equations for $\Psi$).

\begin{prop}[Proposition~11.2 in \cite{KM11}] \label{level2current}
For a current primary field $V$ with charges $q,q_*$ in $\FF_{(b)},$
$$\Big( L_{-2}-\frac1{2q^2} L_{-1}^2\Big )V =0$$
provided $2q(b+q) = 1.$
\end{prop}


\begin{thebibliography}{10}

\bibitem{BBH05}
M.~Bauer, D.~Bernard, and J.~Houdayer.
\newblock Dipolar stochastic {L}oewner evolutions.
\newblock {\em J. Stat. Mech. Theory Exp.}, (3):P03001, 18 pp. (electronic),
  2005.
\newblock \arXiv{math-ph/0411038}.

\bibitem{BB03}
Michel Bauer and Denis Bernard.
\newblock Conformal field theories of stochastic {L}oewner evolutions.
\newblock {\em Comm. Math. Phys.}, 239(3):493--521, 2003.
\newblock \arXiv{hep-th/0210015}.

\bibitem{BB04}
Michel Bauer and Denis Bernard.
\newblock C{FT}s of {SLE}s: the radial case.
\newblock {\em Phys. Lett. B}, 583(3-4):324--330, 2004.
\newblock \arXiv{math-ph/0310032}.

\bibitem{Cardy04}
John Cardy.
\newblock Calogero-{S}utherland model and bulk-boundary correlations in
  conformal field theory.
\newblock {\em Phys. Lett. B}, 582(1-2):121--126, 2004.
\newblock \arXiv{hep-th/0310291}.

\bibitem{FW03}
Roland Friedrich and Wendelin Werner.
\newblock Conformal restriction, highest-weight representations and {SLE}.
\newblock {\em Comm. Math. Phys.}, 243(1):105--122, 2003.
\newblock \arXiv{math-ph/0301018}.

\bibitem{Kang12}
Nam-Gyu Kang.
\newblock Conformal field theory of dipolar {SLE}(4) with mixed boundary
  condition.
\newblock {\em J. Korean Math. Soc.}, 50(4):899--916, 2013.
\newblock \arXiv{1306.6705}.

\bibitem{KM12}
Nam-Gyu Kang and Nikolai~G. Makarov.
\newblock Radial {SLE} martingale-observables.
\newblock 2012.
\newblock Preprint, \arxiv{1208.2789}.

\bibitem{KM11}
Nam-Gyu Kang and Nikolai~G. Makarov.
\newblock Gaussian free field and conformal field theory.
\newblock {\em Ast\'erisque}, (353), 2013.
\newblock \arXiv{1101.1024}.

\bibitem{Kenyon01}
Richard Kenyon.
\newblock Dominos and the {G}aussian free field.
\newblock {\em Ann. Probab.}, 29(3):1128--1137, 2001.
\newblock \arxiv{math-ph/0002027}.

\bibitem{Lawler05}
Gregory~F. Lawler.
\newblock {\em Conformally invariant processes in the plane}, volume 114 of
  {\em Mathematical Surveys and Monographs}.
\newblock American Mathematical Society, Providence, RI, 2005.

\bibitem{LSW03}
Gregory~F. Lawler, Oded Schramm, and Wendelin Werner.
\newblock Conformal restriction: the chordal case.
\newblock {\em J. Amer. Math. Soc.}, 16(4):917--955 (electronic), 2003.
\newblock \arXiv{math/0209343}.

\bibitem{RBGW07}
I.~Rushkin, E.~Bettelheim, I.~A. Gruzberg, and P.~Wiegmann.
\newblock Critical curves in conformally invariant statistical systems.
\newblock {\em J. Phys. A}, 40(9):2165--2195, 2007.
\newblock \arXiv{cond-mat/0610550}.

\bibitem{SW05}
Oded Schramm and David~B. Wilson.
\newblock S{LE} coordinate changes.
\newblock {\em New York J. Math.}, 11:659--669 (electronic), 2005.

\bibitem{Zhan04}
Dapeng Zhan.
\newblock {\em Random {L}oewner Chains in {R}iemann Surfaces}.
\newblock 2004.
\newblock Thesis (Ph.D.)--California Institute of Technology.

\bibitem{Zhan08}
Dapeng Zhan.
\newblock Duality of chordal {SLE}.
\newblock {\em Invent. Math.}, 174(2):309--353, 2008.
\newblock \arXiv{0712.0332}.

\end{thebibliography}



\end{document}